\documentclass[onefignum,onetabnum]{siamart220329} 

\usepackage{subcaption}
\newsiamthm{condition}{Condition}

\usepackage{lipsum}
\usepackage{amsfonts}
\usepackage{graphicx}
\usepackage{epstopdf}
\usepackage{algorithmic}
\ifpdf
  \DeclareGraphicsExtensions{.eps,.pdf,.png,.jpg}
\else
  \DeclareGraphicsExtensions{.eps}
\fi

\newsiamremark{remark}{Remark}
\newsiamremark{hypothesis}{Hypothesis}
\newsiamremark{example}{Example}
\newsiamremark{assumption}{Assumption}
\crefname{assumption}{Assumption}{Assumptions}
\newsiamthm{problem}{Problem}
\crefname{hypothesis}{Hypothesis}{Hypotheses}
\newsiamthm{claim}{Claim}

\headers{Recursive Optimal Stopping with Poisson Stopping Constraints}{G. Liang, W. Wei, Z. Wu, and Z. Xu}

\title{Recursive Optimal Stopping with Poisson Stopping Constraints\thanks{{Funding:} GL was partially funded by NSFC (No.121711169). ZW was funded by the National Key R\&D Program of China (No.2023YFA1009200), the Natural Science Foundation of Shandong Province (No.ZR2019ZD42) and the Taishan Scholars Climbing Program of Shandong (No.TSPD20210302).}}

\author{
    Gechun Liang\thanks{Department of Statistics, University of Warwick, Coventry CV4 7AL, U.K. E-mail:\email{g.liang@warwick.ac.uk}.}
    \and Wei Wei\thanks{Department of Actuarial Mathematics and Statistics, Heriot-Watt University, Edinburgh EH14 4AS, U.K. E-mail:\email{wei.wei@hw.ac.uk}.}
    \and 
    Zhen Wu\thanks{School of Mathematics, Shandong University, Jinan 250100, P.R. China. E-mail:\email{wuzhen@sdu.edu.cn}, \email{zhendaxu@mail.sdu.edu.cn}.}
    \and 
    Zhenda Xu\footnotemark[4]
}

\usepackage{amsopn}

\ifpdf
\hypersetup{
  pdftitle={Recursive Optimal Stopping with Poisson Stopping Constraints},
  pdfauthor={G. Liang, W. We, Z. Wu, and Z. Xu}
}
\fi

\usepackage{comment}
\newcommand{\gl }{\color{black}}

\begin{document}

\maketitle

\begin{abstract}
This paper solves a recursive optimal stopping problem with Poisson stopping constraints using the penalized backward stochastic differential equation (PBSDE) with jumps. Stopping in this problem is only allowed at Poisson random intervention times, and jumps play a significant role not only through the stopping times but also in the recursive objective functional and model coefficients. To solve the problem, we propose a decomposition method based on Jacod-Pham that allows us to separate the problem into a series of sub-problems between each pair of consecutive Poisson stopping times. To represent the value function of the recursive optimal stopping problem when the initial time falls between two consecutive Poisson stopping times and the generator is concave/convex, we leverage the comparison theorem of BSDEs with jumps. We then apply the representation result to American option pricing in a nonlinear market with Poisson stopping constraints.
\end{abstract}

\begin{keywords}
Constrained optimal stopping, recursive objective functional, Penalized backward stochastic differential equation, Poisson stopping times, Jacod-Pham decomposition.
\end{keywords}

\begin{MSCcodes}
60H10, 60G40, 93E20
\end{MSCcodes}

\section{Introduction}
\label{sec:intro}

The solution of a reflected backward stochastic differential equation (RBSDE)
is widely recognized to correspond to the value function of an optimal stopping problem. See \cite{cvitanic1996backward}, \cite{el1997reflected} and \cite{hamadene1999infinite} driven by Brownian motion, \cite{grigorova2017reflected}, \cite{hamadene2003reflected} and \cite{hamadene2016reflected} involving jumps,  and \cite{bayraktar2011optimal}, \cite{Dumitrescu}, \cite{grigorova2020optimal} and \cite{quenez2014reflected} for the case of $g$-expectations.

However, it is worth noting that a penalized backward stochastic differential equation (PBSDE), often utilized to approximate an RBSDE as the penalization parameter $\lambda$ approaches infinity, also allows for stochastic control representations.
Lepeltier and Xu \cite{Lepeltier2005} established a connection between the solution of a PBSDE and the value function of a standard optimal stopping problem. In this connection, a modified obstacle $\min\{S_t, Y_t^\lambda\}$ depending on the original obstacle $S_t$ and the solution $Y_t^\lambda$ of the PBSDE has been introduced. More relevant to our work, Liang \cite{liang2015stochastic} discussed the links between the solution of a PBSDE and the value function of an optimal stopping problem with Poisson stopping constraints. Roughly speaking, he established that the PBSDE (\ref{FBSDE-Penalization-2}) in Lemma \ref{recursive-nonre} admits the following representation:
\begin{equation}\label{linear_case}
Y_0^{\lambda} = \sup_{\tau > 0} \mathbb{E} \left[ \int_0^{\tau \wedge T} f(x, X_s) \, ds + g(\tau \wedge T, X_{\tau \wedge T})\right],
\end{equation}
where \(\tau\) is selected from an exogenous sequence of Poisson stopping times.

The study of optimal stopping problems with Poisson stopping constraints was initiated by Dupuis and Wang \cite{dupuis2002optimal} who approached it by solving two ordinary differential equations defined in the continuation and stopping regions. Subsequently, Lempa \cite{lempa2012optimal} and Hobson \cite{hobson2021shape}  further developed this line of research by considering general one-dimensional diffusions. Liang and Sun \cite{liang2019dynkin} extended to Dynkin games with Poisson stopping constraints and investigated optimal conversion and calling strategies for convertible bonds. Hobson and Zeng \cite{hobson2019constrained} generalized the exogenous homogeneous Poisson process in Poisson stopping constraints to an inhomogeneous Poisson process, allowing players to choose their intensity with the corresponding costs. Menaldi and Robin examined stopping constraints with independent and identically distributed arrival times for optimal stopping problems in \cite{menaldi2016some} and for optimal impulse control problems in \cite{menaldi2017some} and \cite{menaldi2018some}.  Reisinger and Zhang further obtained the convergence rates in \cite{reisinger2020error}.

Poisson stopping constraints offer both mathematical generalizations and valuable modeling merits. The underlying Poisson process is utilized to capture external instantaneous impacts on the system, such as credit events. These events result in totally inaccessible jumps within the system. Additionally, Poisson stopping constraints serve as external constraints that determine admissible stopping times for the system to have sufficient liquidity and be allowed to stop. Furthermore, they can be employed to describe information constraints,  allowing for the consideration of scenarios where observations are only available at Poisson stopping times. For their applications, see Liang et al. \cite{liang2015funding} for dynamic bank run problems; Palmowski et al. \cite{Palmowski2020} for optimal capital structure models.

In this paper, we focus on a recursive optimal stopping problem with Poisson stopping constraints, with the objective functional represented by a sequence of possibly nonlinear BSDEs with jumps, stopped by a sequence of Poisson stopping times. The player then selects over Poisson stopping times to maximize their recursive objective functional. Unlike previous studies on nonlinear optimal stopping under $g$-expectation  (see \cite{bayraktar2011optimal}, \cite{Dumitrescu}, \cite{grigorova2020optimal}, and \cite{quenez2014reflected}), our model does not require the extra constant preservation assumption on the $g$-expectation. Moreover, while existing results on Poisson stopping constraints typically assume that the objective functional depends solely on Brownian filtration, we demonstrate that including Poisson filtration in the objective functional introduces unexpected complexity. Indeed,  when the objective functional is solely based on the Brownian filtration, it results in a conventional penalty term commonly used in PBSDEs. However, when including jumps into the objective functional, this is no longer the case. Our findings shed light on the intricate
challenges associated with including jumps in the objective functional of optimal stopping problems with Poisson stopping constraints.

One of the fundamental tools utilized in our analysis is the Jacod-Pham decomposition for optional and predictable processes in the underlying Brownian and Poisson filtration. This filtration assumption allows us to decompose a process, whether it is optional or predictable, into a sequence of processes defined over consecutive  Poisson stopping time intervals. Within each interval, the corresponding process depends solely on the jumps up to that point and on the Brownian motion. The Jacod-Pham decomposition, initially proposed by Jacod \cite{jacod1985grossissement} (see also Jeulin and Yor \cite{jeulin2006grossissements}) has been subsequently applied to solve BSDEs with jumps by Pham \cite{pham2010stochastic}  (see also Jiao et al. \cite{jiao2013}, Kharroubi and Lim \cite{kharroubi2014progressive}). Motivated by the Jacod-Pham decomposition, we introduce an update scheme for the penalty term within each Poisson stopping time interval. This scheme involves shifting the value function and the obstacle process forward. By utilizing the representation formula of the jump component of the BSDE solution and introducing a new auxiliary process, we are able to capture and record the differences of the value functions in any two consecutive  Poisson stopping time intervals as well as the differences in the obstacle processes within these intervals.
By incorporating these elements, we effectively account for the variations and changes that occur across successive Poisson stopping time intervals.

Regarding the generator in the BSDE representation for the recursive objective functional, we first consider the case that the generator is independent of $Z$ and $C$, which are the hedging components of the BSDE solution. In this case, the objective functional is defined recursively using the solution $Y$ of the representing BSDE. The corresponding PBSDE representation for the recursive optimal stopping problem with Poisson stopping constraints has a relatively simple yet interpretable penalty term. This penalty term involves the forward shift of the value function and the obstacle process. However, when the generator depends on $Z$ and $C$, the penalty term takes on an unexpectedly complex form due to the intricacies caused by the Poisson stopping times. In the latter case, we first consider the scenario where the generator depends on $Z$ and $C$ in a linear way. We introduce an adjoint process allowing us to transform the generator in such a way that it no longer relies on $Z$ and $C$. When the generator is concave/convex in $Z$ and $C$, we apply a convex duality argument along with the BSDE comparison with jumps and propose that the corresponding PBSDE representation for the recursive optimal stopping with Poisson stopping constraints should be a limit for a sequence of PBSDEs associated with the linear dependence on $Z$ and $C$.

Many examples of objective functionals fit into our recursive optimal stopping framework. One such example is the pricing of American options, as proposed in \cite{dupuis2002optimal}. In this case, the generator is linear in $Y$ and $Z$, while remaining independent of $C$. Another example is the stock trading model presented in \cite{DeAngelis2022recursive}, which is equivalent to the generator that depends solely on $Y$. In the following, we address a comprehensive example in American options that includes all possible cases we investigate in the paper.

\begin{example}[American option pricing in a nonlinear market with Poisson stopping constraints]\label{ep-2}
Suppose that a market consists of three securities: a risk-free bond $S^0$ and two risky assets $S^1$ and $S^2$. They evolve according to the following equations
\begin{align*}
dS_t^0&=S_t^0r_tdt,\\
dS_t^1&=S_t^1[\mu_t^1dt+\sigma_tdW_t],\\
dS_t^2&=S_{t-}^2[\mu_t^2dt+\eta_t\tilde{N}(dt)],
\end{align*}
where $W$ is a standard Brownian Motion, $\tilde{N}$ is the compensated Poisson martingale measure of a Poisson process $N$, and $r_t\in\mathbb{R}, \mu_t^1\in\mathbb{R},\mu_t^2\in\mathbb{R}, \sigma_t> 0$, $\eta_t>-1$ are bounded processes predictable to the underling Brownian motion and Poisson process. To exclude arbitrage, we further assume that $\mu_t^2=0$ when $\eta_t=0$. Otherwise, $S^2$ becomes another risk-free bond. Furthermore, for technical reasons, we assume that $1/\sigma_t$ and $1/\eta_t\mathbb{I}_{\{\eta_t\neq 0\}}$ are also bounded.

Let $\pi=(\pi^1,\pi^2)$ denote the amount of the money invested in the risky assets $(S^1,S^2)$. The self-financing strategy leads to a linear equation for the wealth process
\begin{equation}\label{wealth_equ}
dV_t=\big[r_tV_t+ (\mu_t^1-r_t){\pi}_t^1+(\mu_t^2-r_t)\pi_t^2\big]dt + \sigma_t \pi_t^1 dW_t + \eta_t\pi_t^2 \tilde{N}(dt),
\end{equation}

One may generalize the above linear wealth equation to a general nonlinear case, and use it to replicate an American option with payoff $g(S^1,S^{2})$ and maturity $T$, exercised at a sequence of Poisson stopping times $\tau$ generated by the Poisson process $N$.
It will lead to a sequence of BSDEs with jumps as the objective functional, stopped by a sequence of Poisson stopping times $\tau$:
\[
\tilde{Y}_{t}^\tau = g({S}^1_{\tau\wedge T},{S}^2_{\tau\wedge T}) + \int_{t\wedge \tau}^{\tau\wedge T} f(s, \tilde{Y}_s^\tau, \tilde{Z}_s^\tau, \tilde{C}_s^\tau)ds- \int_{t\wedge \tau}^{\tau\wedge T} \tilde{Z}_s^\tau dW_s- \int_{t\wedge \tau}^{\tau\wedge T} \tilde{C}_s^\tau\tilde{N}(ds).
\]
The player then selects over the sequence of Poisson stopping times $\tau$ to maximize their recursive objective functional:  $\sup_{\tau>0} \tilde{Y}^\tau_0.$  In the classical linear case, the generator $f$ has the form
$$f(t,y,z,c)=-r_ty-\frac{\mu_t^1-r_t}{\sigma_t}z-\frac{\mu_t^2-r_t}{\eta_t}c\mathbb{I}_{\{\eta_t\neq 0\}}$$
according to the linear wealth equation (\ref{wealth_equ}).

In this paper, we will discuss general nonlinear cases motivated by American option models in nonlinear markets (see, for example, \cite{Dumitrescu2}, \cite{grigorova2021american}, and \cite{Tianyang}). This is driven by the existence of bid-ask spreads for the drifts of the two risky assets and the interest rate of the risk-free bond, which depend on the long-short positions of the two risky assets and the risk-free bond. Specifically, we will assume that $(\mu_t^1,\mu_t^2,r_t)=(\underline{\mu}^1_t,\underline{\mu}^2_t,\underline{r}_t)$ for the long positions of the two risky assets and the risk-free bond, and $(\mu_t^1,\mu_t^2,r_t)=(\overline{\mu}^1_t,\overline{\mu}^2_t,\overline{r}_t)$ for the short positions. The bounded and predictable processes $\overline{\mu}^1_t\geq \underline{\mu}^1_t$, $\overline{\mu}^2_t\geq \underline{\mu}^2_t$, and $\overline{r}_t\geq \underline{r}_t$ represent the bid-ask spreads of the drifts of the two risky assets and the interest rate of the risk-free bond, respectively. Consequently, the linear wealth equation (\ref{wealth_equ}) becomes nonlinear:     \begin{align}\label{wealth_equ_nonlinear}
dV_t
=& \big[(V_t - \pi_t^1 - \pi_t^2)^+ \underline{r}_t - (V_t - \pi_t^1 - \pi_t^2)^- \overline{r}_t\big] dt \notag\\
& + (\pi_t^1)^+ (\underline{\mu}^1_t dt + \sigma_t dW_t) - (\pi_t^1)^- (\overline{\mu}^1_t dt + \sigma_t dW_t) \notag\\
& + (\pi_t^2)^+ (\underline{\mu}^2_t dt + \eta_t \tilde{N}(dt)) - (\pi_t^2)^- (\overline{\mu}^2_t dt + \eta_t \tilde{N}(dt)) \notag\\
=& \big[\underline{r}_t (V_t - \pi_t^1 - \pi_t^2)^+ - \overline{r}_t (V_t - \pi_t^1 - \pi_t^2)^- \notag\\
&\quad + \underline{\mu}^1_t (\pi_t^1)^+ - \overline{\mu}^1_t (\pi_t^1)^- \notag\\
&\quad + \underline{\mu}^2_t (\pi_t^2)^+ - \overline{\mu}^2_t (\pi_t^2)^-\big] dt+ \sigma_t \pi_t^1 dW_t + \eta_t \pi_t^2 \tilde{N}(dt).
\end{align}
It then follows from (\ref{wealth_equ_nonlinear}) that the generator has the following nonlinear form:
\begin{align*}
f(t, y, z, c) = & -\underline{r}_t \left(y - \frac{z}{\sigma_t} - \frac{c}{\eta_t}\mathbb{I}_{\{\eta_t\neq 0\}}\right)^+
                  + \overline{r}_t \left(y - \frac{z}{\sigma_t} - \frac{c}{\eta_t}\mathbb{I}_{\{\eta_t\neq 0\}}\right)^- \\
                & - \underline{\mu}^1_t \frac{z^+}{\sigma_t}
                  + \overline{\mu}^1_t \frac{z^-}{\sigma_t}
                  - \underline{\mu}^2_t \left(\frac{c}{\eta_t}\right)^+\mathbb{I}_{\{\eta_t\neq 0\}}
                  + \overline{\mu}^2_t \left(\frac{c}{\eta_t}\right)^-\mathbb{I}_{\{\eta_t\neq 0\}},
\end{align*}
which is Lipchitz continuous and convex in $(y,z,c)$.
\end{example}

The rest of this paper is organized as the following. \Cref{sec:preli} formulates the recursive optimal stopping problem with Poisson stopping constraints and introduces a class of PBSDEs with jumps as the candidate solutions for the problem. Then, we prove the representation results by considering the case where the generator is independent of $Z$ and $C$ in \Cref{sec:repre-noZC}, linear in $Z$ and $C$ in \Cref{sec:repre-ZC} and convex in $Z$ and $C$ in \Cref{sec:repre-ZC-nonlinear}.  We apply the representation results to American option pricing in \Cref{sec:appli} and conclude in \Cref{sec:concl} with possible extensions.

\section{Preliminary and main result}
\label{sec:preli}

In this section, we present the formulation of the recursive optimal stopping problem with Poisson stopping constraints and introduce a class of PBSDEs used to characterize its solution.

Let $\left(\Omega,\mathcal{F},\mathbb{P}\right)$ be a complete probability space supporting a $d$-dimensional standard Brownian motion $W$ with its augmented filtration $\mathbb{F}=\left\{\mathcal{F}_t\right\}_{t\geq 0}$. Moreover, let $\left\{T_i\right\}_{i\geq 0}$ be the arrival times of an independent Poisson process $N$ with the intensity $\lambda$ and the augmented filtration $\mathbb{H} =\left\{\mathcal{H}_t\right\}_{t\geq 0}$. Define $\mathcal{G}_t := \mathcal{F}_t\bigvee\mathcal{H}_t$ and $\mathbb{G} := \left(\mathcal{G}_t\right)_{t\geq 0}$. Without loss of generality, we assume that $T_0=0$ and $T_\infty = \infty$. We also use $N := N(\omega, dt)$ to represent the random jump measure generated by the Poisson process $N$ with the compensator $\hat{N}(dt) := \lambda dt$ such that $\tilde{N}((0,t]) = (N-\hat{N})((0,t])$, $t\geq 0$, is a martingale. In the rest of our paper, we will use the following notations:

\begin{align*}
L_{\mathbb{G}}^2(0,T;\mathbb{R}^k) &:= \Big\{ \varphi : [0,T]\times \Omega \rightarrow \mathbb{R}^k\ \Big|\ \{\varphi_t\}_{0\leq t\leq T}\text{ is }\mathbb{G}\text{-progressively} \\
&\qquad \quad \text{measurable with }\mathbb{E} \Big[ \int_0^T \left|\varphi_t\right|^2dt \Big] < +\infty  \Big\},\\
L_{\mathbb{G}}^{2,p}(0,T;\mathbb{R}^k) &:= \Big\{ \varphi : [0,T]\times \Omega \rightarrow \mathbb{R}^k\ \Bigg|\ \{\varphi_t\}_{0\leq t\leq T}\text{ is }\mathbb{G}\text{-predictable with }\\
&\qquad \quad \mathbb{E} \Big[ \int_0^T \left|\varphi_t\right|^2dt \Big] < +\infty  \Big\},\\
%
%
%
S_{\mathbb{G}}^2(0,T;\mathbb{R}^k) &:= \Big\{ \varphi : [0,T]\times \Omega \rightarrow \mathbb{R}^k\ \Big|\ \{\varphi_t\}_{0\leq t\leq T}\text{ is }\mathbb{G}\text{-adapted with}\\
&\qquad \quad \text{c\`adl\`ag sample paths and } \mathbb{E} \Big[ \sup_{0\leq t\leq T}\left|\varphi_t\right|^2 \Big] < +\infty  \Big\},\\
\mathcal{T} &:= \Big\{ \tau : \Omega \rightarrow [0,T]\ \Big|\ \tau \text{ is a $\mathbb{G}$-stopping time}\Big\},\\
\mathbb{I}_t^i &:= \mathbb{I}_{\{T_{i-1} \leq t < T_i\}},\quad  i\geq 1.
\end{align*}

\subsection{Recursive optimal stopping problem with Poisson stopping constraints}

To start with, we introduce the admissible stopping set starting at $T_i$ as
\begin{align*}
\mathcal{R}_{T_i}(\lambda) := \left\{\left. \tau\in \mathcal{T}\, \right| \,\text{for any }\omega\in\Omega,\,\tau(\omega) = T_n(\omega)\wedge T\text{, where }n \geq i \right\},\
\end{align*}
and the admissible stopping set starting at time $t$, which may not necessarily be a jump time, as
\begin{align*}
\mathcal{R}_{t}(\lambda) := \left\{\left. \tau\in \mathcal{T}\, \right| \,\text{for any }\omega\in\Omega,\,\tau(\omega) = T_n(\omega)\wedge T \geq t\text{, where }n \geq 0 \right\}.
\end{align*}
For $\mathcal{R}_{t+}(\lambda)$, we require $\tau(\omega) = T_n(\omega)\wedge T > t$ instead. Moreover, for $\tau \in \mathcal{R}_{t+}(\lambda)$, we introduce the following BSDE with jumps as a recursive objective functional: for $t\in[0,T]$,
\begin{align}\label{FBSDE-Stopping}
\left\{
\begin{aligned}
{X}_t &= x + \int_0^{t} b(s, {X}_s)ds + \int_0^{t} \sigma(s, {X}_s) dW_s + \int_0^{t} h(s, {X}_{s-}) \tilde{N}(ds), \\
\tilde{Y}_{t}^\tau &= {g({\tau}, {X}_{\tau})} + \int_{t}^{\tau} f(s, {X}_s, \tilde{Y}_s^\tau, \tilde{Z}_s^\tau, \tilde{C}_s^\tau)ds \\
&\qquad - \int_{t}^{\tau} \tilde{Z}_s^\tau dW_s- \int_{t}^{\tau} \tilde{C}_s^\tau\tilde{N}(ds).
\end{aligned}
\right.
\end{align}
In addition, the coefficients $b\left(t,x\right) : \Omega \times \left[0,T\right] \times \mathbb{R}^n \rightarrow \mathbb{R}^n$, $\sigma\left(t,x\right) : \Omega \times \left[0,T\right] \times \mathbb{R}^n \rightarrow \mathbb{R}^{n\times d}$, $h\left(t,x\right) : \Omega \times \left[0,T\right] \times \mathbb{R}^n \rightarrow \mathbb{R}^n$, $g\left(t,x\right) : \Omega \times \left[0,T\right] \times \mathbb{R}^n \rightarrow \mathbb{R}$ and $f\left(t,x,y,z,c\right) : \Omega \times \left[0,T\right] \times \mathbb{R}^n \times \mathbb{R} \times \mathbb{R}^d \times \mathbb{R} \rightarrow \mathbb{R}$ satisfy the following assumptions.
\begin{assumption}\label{(H1)}
The coefficients $b(\cdot,x)\in L_{\mathbb{G}}^2(0,T;\mathbb{R}^n),\, \sigma(\cdot,x) \in L_{\mathbb{G}}^2(0,T;\mathbb{R}^{n\times d})$ and $h(\cdot,x) \in L_\mathbb{G}^{2,p}(0,T;\mathbb{R}^n)$ for any $x\in\mathbb{R}^n$, and there exists a constant $L > 0$ such that for any $t\in[0,T]$, $x$ and $x^\prime \in \mathbb{R}^n$,
\begin{align*}
\left|b\left(t,x\right)\right| + \left|\sigma\left(t,x\right)\right| + \left|h\left(t,x\right)\right|  &\leq L\left(1 + |x|\right),\\
\left|b\left(t,x\right) - b\left(t,x^\prime\right)\right| + \left|\sigma\left(t,x\right) - \sigma\left(t,x^\prime\right)\right| + \left|h\left(t,x\right) - h\left(t,x^\prime\right)\right| &\leq L\left|x - x^\prime\right|.
\end{align*}
\end{assumption}
\begin{assumption}\label{(H2)}
The generator $f(\cdot,x,y,z,c) \in L_{\mathbb{G}}^2(0,T;\mathbb{R})$, the obstacle process $g(\cdot,x) \in S_{\mathbb{G}}^2(0,T;\mathbb{R})$ for any $(x,y,z,c)\in \mathbb{R}^n\times \mathbb{R}\times \mathbb{R}^d\times \mathbb{R}$, and there exists a constant $L > 0$ such that for any $t\in[0,T]$, $(x,y,z,c)$ and $(x^\prime, y^\prime, z^\prime, c^\prime) \in \mathbb{R}^n\times \mathbb{R}\times \mathbb{R}^d\times \mathbb{R}$,
\begin{align*}
& \left|g\left(t,x\right)\right|\leq L(1+|x|),\qquad\left|g\left(t,x\right) - g\left(t,x^\prime\right)\right| \leq L\left|x - x^\prime\right|,\\
&\left|f\left(t,x,y,z,c\right)\right| \leq L\left(1 + |x| + |y| + |z|\right),\\
&\left|f\left(t,x,y,z,c\right)- f\left(t,x^\prime,y^\prime,z^\prime,c^\prime\right)\right| \leq L\left(\left|x - x^\prime\right| + \left|y - y^\prime\right| + \left|z - z^\prime\right| + \left|c - c^\prime\right|\right).
\end{align*}
\end{assumption}
The solvability of \eqref{FBSDE-Stopping} can be obtained under the above assumptions according to \cite[Theorem 2.1]{wu2003fully}.
\begin{lemma}\label{BSDE-RandomDuration}
Suppose that \cref{(H1),(H2)} hold. The BSDE with jumps \eqref{FBSDE-Stopping} admits a unique adapted solution $({X}, \tilde{Y}^\tau, \tilde{Z}^\tau, \tilde{C}^\tau) \in S_\mathbb{G}^2(0, T;\mathbb{R}^n) \times S_\mathbb{G}^2(0, T;\mathbb{R}) \times L_\mathbb{G}^{2,p}(0, T;\mathbb{R}^d)\times L_\mathbb{G}^{2,p}(0,T;\mathbb{R}) $.
\end{lemma}

The BSDE with jumps (\ref{FBSDE-Stopping}) requires an additional condition, as stated in \cite[Theorem 4.1]{wu2003fully} and \cite[Theorem 4.4]{quenez2013bsdes}, in order to restrict the magnitude of the jumps and ensure that the comparison theorem holds.
\begin{assumption}\label{(H3)}
For any $(t,y,z)\in [0,T]\times \mathbb{R}\times \mathbb{R}^d,\,c$ and $c^\prime\in \mathbb{R}$ with $c \neq c^\prime$,
\begin{align}\label{assumption-compari}
\frac{f(t, X_t, y, z, c) - f(t, X_t, y, z, c^\prime)}{\lambda(c - c^\prime)} \geq -1.
\end{align}
\end{assumption}

The recursive objective functional of the optimal stopping problem with Poisson stopping constraints is given by \eqref{FBSDE-Stopping}. Then, we propose the following optimal stopping problem.
\begin{problem*}[Recursive optimal stopping problem with Poisson stopping constraints]\label{(R-OSP-P)}
Suppose that \cref{(H1),(H2),(H3)} hold. We aim to find an admissible stopping time $\hat{\tau}^* \in \mathcal{R}_{T_1}(\lambda)$ to maximize the recursive objective functional given by \eqref{FBSDE-Stopping}, such that
\begin{align*}
V^\lambda_0: = \sup_{\tau\in \mathcal{R}_{T_1}(\lambda)} \tilde{Y}_0^\tau  = \tilde{Y}_0^{\hat{\tau}^*}.
\end{align*}
In general, $V_t^\lambda$ is defined as the value function of the optimal stopping problem
\begin{align*}
V^\lambda_t := \sup_{\tau\in \mathcal{R}_{T_i}(\lambda)} \tilde{Y}_t^\tau=\tilde{Y}_t^{\hat{\tau}^*_t} , \quad \text{conditional on }\big\{T_{i-1}\leq t<T_i\big\}, \quad i \geq 1,
\end{align*}
with the optimal stopping time denoted as $\hat{\tau}_t^*$.
\end{problem*}
 {\gl Note that when taking the supremum over a family of random variables, we mean the essential supremum, though we do not explicitly distinguish between them in our notation.} Additionally, we have
 $V^\lambda_t= \sup_{\tau\in \mathcal{R}_{t+}(\lambda)} \tilde{Y}_t^\tau$ for $t\in[0,T)$ and, in particular, when $t=0,$ $\tau\geq T_1$, so the system is not allowed to stop at initial time $0$.\footnote{At this point, we explain the notations used in the paper. We use \(\tilde{Y}^{\tau}\) to represent the recursive objective functional given by BSDE \eqref{FBSDE-Stopping}, \(V^{\lambda}\) to denote the value function of the recursive optimal stopping problem \ref{(R-OSP-P)}, \(\hat{Y}^{\lambda}\) to represent the solution of the characterizing PBSDE \eqref{FBSDE-Penalization}, and finally \(Y^{\lambda}\) to represent the value function in \eqref{linear_case} and the solution of PBSDE \eqref{FBSDE-Penalization-2}.}

\subsection{PBSDE with jumps} 

In the following, we introduce a new PBSDE with jumps which will be utilized to solve \cref{(R-OSP-P)}, namely, for $t\in[0, T]$,
\begin{align}\label{FBSDE-Penalization}
\begin{aligned}
\hat{Y}_t^\lambda &= g({ T}, {X}_{ T}) + \int_t^{T} \Big[f\Big(s, {X}_s, \hat{Y}_s^\lambda, \hat{Z}_s^\lambda, \hat{C}_s^\lambda+\Big( g(s,{X}_s) + A_s - \hat{Y}_{s}^\lambda- \hat{C}_{s}^\lambda \Big)^+\Big)\\
&\qquad + \lambda \Big( g(s,{X}_s) + A_s - \hat{Y}_{s}^\lambda- \hat{C}_{s}^\lambda \Big)^+\Big] ds- \int_t^{ T} \hat{Z}_s^\lambda dW_s - \int_{t}^T \hat{C}_s^\lambda \tilde{N}(ds),
\end{aligned}
\end{align}
 where $A\in L_{\mathbb{G}}^{2}(0,T;\mathbb{R})$ is
an auxiliary process used to record the difference of $g(t,{X}_t)$ in $[T_{i-1}, T_i)$ and $[T_i, T_{i+1})$ for $i\geq 1$. More specifically, we introduce (Poisson) time indexes
\begin{align*}
\begin{aligned}
T_{(i)} &:= (T_1, \cdots, T_i), \quad i\geq 1,\\
\Delta_i &:= \left\{ \theta_{(i)} := (\theta_1, \cdots, \theta_i) \in\mathbb{R}^i \left| \theta_1\leq \theta_2 \leq \cdots \leq \theta_{i} \right. \right\}, \quad i\geq 1.
\end{aligned}
\end{align*}
We also write $T_{(i)} = (T_{(i-1)}, T_i)$ and $\theta_{(i)} = (\theta_{(i-1)}, \theta_i)$. Then, the Jacod-Pham decomposition for optional processes (see \cite{jacod1985grossissement}, \cite{jeulin2006grossissements} and more recently \cite{pham2010stochastic}) implies that
\begin{equation*}
g(t, {X}_{t}) = \sum_{i \geq 0} g^{i}(t, T_{(i)}) \mathbb{I}_t^{i+1},
\end{equation*}
where $g^i(\cdot,\theta_{(i)})\in L_{\mathbb{F}}^2(0,T;\mathbb{R})$ for any $\theta_{(i)}\in \Delta_i$. In turn, the auxiliary process $A$ is defined as
\begin{align*}
\begin{aligned}
{A}_t &:= \sum_{i \geq 0} {A}^{i}(t, T_{(i)}) \mathbb{I}_t^{i+1},\\
\end{aligned}
\end{align*}
with ${A}^{i}(t, T_{(i)}) := {A}^{i}(t, \theta_{(i)})|_{\theta_{(i)} = T_{(i)}}$ given by $$ {A}^{i}(t, \theta_{(i)}) := g^{i+1}(t, (\theta_{(i)}, t)) - g^{i}(t, \theta_{(i)}),\quad i\geq 0.$$

It is evident that when the  generator $f$ satisfies Assumption \ref{(H2)}, the generator of PBSDE \eqref{FBSDE-Penalization} also satisfies the same assumption. Following the result in \cite[Theorem 2.1]{wu2003fully}, we have:
\begin{lemma}\label{BSDE-p}
Suppose that \cref{(H1),(H2)} hold. Then, the PBSDE with jumps \eqref{FBSDE-Penalization} admits a unique adapted solution $( \hat{Y}^\lambda, \hat{Z}^\lambda, \hat{C}^\lambda) \in S_\mathbb{G}^2(0, T;\mathbb{R}) \times L_\mathbb{G}^{2,p}(0, T;\mathbb{R}^d)\times L_\mathbb{G}^{2,p}(0,T;\mathbb{R})$.
\end{lemma}

In addition, note that if the generator $f$ satisfies Assumption \ref{(H3)}, then the generator of  \eqref{FBSDE-Penalization} also satisfies the same assumption. In other words, for any $(t,y,z)\in [0,T]\times \mathbb{R}\times \mathbb{R}^d,\,c$ and $c^\prime\in \mathbb{R}$ with $c \neq c^\prime$, the generator meets the requirements of Assumption \ref{(H3)}:
\begin{align*}
\begin{aligned}
&\frac{f(t, X_t, y, z, c+(g(t,X_t)+A_t-y-c)^+) +\lambda (g(t,X_t)+A_t-y-c)^+}{\lambda(c - c^\prime)}\\
&-\frac{f(t, X_t, y, z, c'+(g(t,X_t)+A_t-y-c^\prime)^+) +\lambda (g(t,X_t)+A_t-y-c^\prime)^+)}{\lambda(c - c^\prime)} \geq -1,
\end{aligned}
\end{align*}
so the comparison theorem also holds for \eqref{FBSDE-Penalization}.

\subsection{Decomposition for PBSDE with jumps}
This part provides a decomposition result for the coefficients and the solution of the PBSDE with jumps \eqref{FBSDE-Penalization}, which plays an important role in deriving its representation for the value function $V^\lambda_t$ of \cref{(R-OSP-P)}.

First, by employing the definition of the martingale measure $\tilde{N}$, we can establish the equivalence of (\ref{FBSDE-Penalization}) with the following equation:
\begin{align*}
\begin{aligned}
\hat{Y}_t^\lambda &= g({ T}, {X}_{ T}) + \int_t^{T} f\Big(s, {X}_s, \hat{Y}_s^\lambda, \hat{Z}_s^\lambda, \hat{C}_s^\lambda+\Big( g(s,{X}_s) + A_s - \hat{Y}_{s}^\lambda- \hat{C}_{s}^\lambda \Big)^+\Big)ds\\
&\qquad + \int_t^{T}\lambda \max\Big\{ g(s,{X}_s) + A_s - \hat{Y}_{s}^\lambda,  \hat{C}_{s}^\lambda\Big\} ds- \int_t^{ T} \hat{Z}_s^\lambda dW_s - \int_{t}^T \hat{C}_s^\lambda {N}(ds).
\end{aligned}
\end{align*}
Apply the Jacod-Pham decomposition to the generator of the above PBSDE , we obtain
\begin{align*}\label{decompo-func}
&f\Big(t,{X}_{t},y,z,c+\Big(g(t,X_t)+A_t-y-c\Big)^+\Big) +\lambda \max\Big\{ g(t,{X}_t) + A_t - y,  c\Big\} \\
&= \sum_{i \geq 0} h^{i,\lambda}(t,y,z,c, T_{(i)}) \mathbb{I}_t^{i+1},
\end{align*}
where $h^{i,\lambda}(\cdot,y,z,c,\theta_{(i)}) \in L_{\mathbb{F}}^2(0,T;\mathbb{R})$ for any $(y,z,c,\theta_{(i)}) \in \mathbb{R} \times \mathbb{R}^d \times \mathbb{R}\times \Delta_i$. Then, PBSDE \eqref{FBSDE-Penalization} admits the following decomposition.
\begin{theorem}\label{Decompo-Y-lambda}
The solution $(\hat{Y}^\lambda, \hat{Z}^\lambda, \hat{C}^\lambda)$ of the PBSDE with jumps \eqref{FBSDE-Penalization} has the following representation for any $t\in[0,T]$:
\begin{align*}
\hat{Y}^\lambda_t &=\sum_{i \geq 0} \hat{Y}^{i,\lambda}_t (T_{(i)}) \mathbb{I}_t^{i+1},\\
\hat{Z}^\lambda_t &= \hat{Z}^{0,\lambda}_t \mathbb{I}_{\{t \leq T_1\}} + \sum_{i \geq 1} \hat{Z}^{i,\lambda}_t (T_{(i)}) \mathbb{I}_{\{T_i < t \leq T_{i+1}\}},\\
\hat{C}^\lambda_t &= \hat{C}^{0,\lambda}_t \mathbb{I}_{\{t \leq T_1\}} + \sum_{i \geq 1} \hat{C}^{i,\lambda}_t (T_{(i)}) \mathbb{I}_{\{T_i < t \leq T_{i+1}\}},
\end{align*}
with the processes
\begin{align*}
\hat{Y}^{i,\lambda}_t (T_{(i)}) &:= \hat{Y}^{i,\lambda}_t (\theta_{(i)})|_{\theta_{(i)} = T_{(i)}},\\
\hat{Z}^{i,\lambda}_t (T_{(i)}) &:= \hat{Z}^{i,\lambda}_t (\theta_{(i)})|_{\theta_{(i)} = T_{(i)}},\\
\hat{C}^{i,\lambda}_t (T_{(i)}) &:= \hat{C}^{i,\lambda}_t (\theta_{(i)})|_{\theta_{(i)} = T_{(i)}},
\end{align*}
given by the following Brownian motion driven BSDE in the time horizon $[\theta_i\wedge T, T]$ with $ \hat{C}_t^{i,\lambda}(\theta_{(i)}) := \hat{Y}_t^{i+1,\lambda}((\theta_{(i)}, t)) - \hat{Y}_t^{i,\lambda}(\theta_{(i)})$ for $i \geq 0$:
\begin{align*}
\begin{aligned}
\hat{Y}_{t}^{i,\lambda} (\theta_{(i)}) &= g^i(T, \theta_{(i)}) + \int_{t}^{T} {h}^{i,\lambda}(s, \hat{Y}_s^{i,\lambda}(\theta_{(i)}), \hat{Z}_s^{i,\lambda}(\theta_{(i)}), \hat{C}_s^{i,\lambda}(\theta_{(i)}), \theta_{(i)})ds\\
&\qquad - \int_{t}^{T} \hat{Z}_s^{i,\lambda}(\theta_{(i)}) dW_s.
\end{aligned}
\end{align*}
\end{theorem}

{\gl
\begin{proof}
First, we have the following decomposition for $\hat{Y}^{\lambda}
_t$ according to the Jacod-Pham decomposition:
\begin{align*}
\hat{Y}^{\lambda}_t &= \hat{Y}^{\lambda}_0
+ \sum_{i\geq 0} \left( \hat{Y}^{\lambda}_{t\wedge T_{i+1}-}
- \hat{Y}^{\lambda}_{t\wedge T_{i}} \right)
+ \sum_{i\geq 0} \left( \hat{Y}^{\lambda}_{t\wedge T_{i+1}}
- \hat{Y}^{\lambda}_{t\wedge T_{i+1}-} \right) \\
&= \hat{Y}^{\lambda}_0
+ \sum_{i\geq 0} \left( \hat{Y}^{i,\lambda}_{t\wedge T_{i+1}}(T_{(i)})
- \hat{Y}^{i,\lambda}_{t\wedge T_{i}}(T_{(i)}) \right) \\
&\quad + \sum_{i\geq 0} \left( \hat{Y}^{i+1,\lambda}_{T_{i+1}}(T_{(i+1)})
- \hat{Y}^{i,\lambda}_{T_{i+1}}(T_{(i)}) \right)
\mathbb{I}_{\{T_{i+1}\leq t\}} \\
&= \hat{Y}^{\lambda}_0 + (I) + (II).
\end{align*}

For each term in (I), using the BSDE for $\hat{Y}^{i,\lambda}(\theta_{(i)})$, we obtain
\begin{align*}
&\hat{Y}^{i,\lambda}_{t\wedge T_{i+1}}(\theta_{(i)})
- \hat{Y}^{i,\lambda}_{t\wedge T_{i}}(\theta_{(i)}) \\
&\quad = -\int^{t\wedge T_{i+1}}_{t\wedge T_{i}}
    {h}^{i,\lambda} \big( s, \hat{Y}_s^{i,\lambda}(\theta_{(i)}),
    \hat{Z}_s^{i,\lambda}(\theta_{(i)}),
    \hat{C}_s^{i,\lambda}(\theta_{(i)}), \theta_{(i)} \big) ds \\
&\qquad + \int^{t\wedge T_{i+1}}_{t\wedge T_{i}}
    \hat{Z}_s^{i,\lambda}(\theta_{(i)}) dW_s.
\end{align*}
By summing up over $i\geq0$, we further obtain
\begin{align*}
(I)&=-\int_0^t f\Big(s, {X}_s, \hat{Y}_s^\lambda, \hat{Z}_s^\lambda, \hat{C}_s^\lambda+\Big( g(s,{X}_s) + A_s - \hat{Y}_{s}^\lambda- \hat{C}_{s}^\lambda \Big)^+\Big)ds\\
&\qquad + \int_t^{T}\lambda \max\Big\{ g(s,{X}_s) + A_s - \hat{Y}_{s}^\lambda,  \hat{C}_{s}^\lambda\Big\} ds- \int_t^{ T} \hat{Z}_s^\lambda dW_s.
\end{align*}
On the other hand,  for term (II), using the definition of $\hat{C}^{i,\lambda}$, we have
\begin{align*}
(II)&=\int_0^{t}\sum_{i\geq 0}\mathbb{I}_{\{T_i<s\leq T_{i+1}\}}\left( \hat{Y}^{i+1,\lambda}_{s}((T_{(i)},s))
- \hat{Y}^{i,\lambda}_{s}(T_{(i)}) \right)N(ds)\\\
&=\int_0^t\sum_{i\geq 0}\mathbb{I}_{\{T_i<s\leq T_{i+1}\}}\hat{C}_s^{i,\lambda}(T_{(i)})ds=\int_0^t\hat{C}^{\lambda}_sN(ds).
\end{align*}
The conclusion then follows by adding up the above two terms (I) and (II).
\end{proof}
}

\begin{remark}\label{recursive-nonre}
When $f(t,x,y,z,c) = f(t,x)$, $h(t,x)=0$ and all the coefficients in \eqref{FBSDE-Stopping} are $\mathbb{F}$-adapted, the solution ${X}$ of the forward equation in \eqref{FBSDE-Stopping} will be $\mathbb{F}$-adapted and the Jacod-Pham decomposition of $g(t, {X}_t)$ will be trivial with $g^i(t,\theta_{(i)}) = g(t, {X}_t)$. Then,
\eqref{FBSDE-Penalization} degenerates to a standard PBSDE on $\left(\Omega,\mathcal{F}, \mathbb{F},\mathbb{P}\right)$, namely,
\begin{align}\label{FBSDE-Penalization-2}
\begin{aligned}
%
{Y}_t^{\lambda} &= g({ T}, {X}_{ T}) + \int_t^{T} f(s, {X}_s)ds + \int_t^{ T} \lambda\left( g(s,{X}_s)-{Y}_s^{\lambda} \right)^+ ds\\
&\qquad- \int_t^{ T} {Z}_s^{\lambda} dW_s,\quad t\in[0, T].
\end{aligned}
\end{align}
\end{remark}

\begin{remark} Unless there are only a finite number of jumps as considered in \cite{pham2010stochastic}, \cite{jiao2013}, and \cite{kharroubi2014progressive}, so one can solve the sequence of Brownian motion driven BSDEs recursively in a backward way, the Jacod-Pham decomposition in Theorem \ref{Decompo-Y-lambda} does not help to solve the PBSDE with jumps \eqref{FBSDE-Penalization}, as the sequence of Brownian motion driven BSDEs is not a closed system therein. Instead, the solvability of \eqref{FBSDE-Penalization} is provided in Lemma \ref{BSDE-p}, and we will utilize the Jacod-Pham decomposition to establish the connection between \eqref{FBSDE-Penalization} and the recursive optimal stopping problem \ref{(R-OSP-P)}.
\end{remark}

\subsection{Main result}
Suppose that $f(t,x,y,z,c)$ is convex in $(y,z,c)$. The concave case can be treated in a similar way.  Introduce the convex dual of the generator $f$:
\begin{equation*}
    f^*(t,X_t,\alpha,\beta,\gamma)=\sup_{(y,z,c)\in\mathbb{R}\times\mathbb{R}^{d}\times\mathbb{R}}\Big\{\alpha y+
    \beta z+\gamma c - f(t,X_t,y,z,c)\Big\}.
\end{equation*}
Since $f(t,x,y,z,c)$ is continuous and convex in $(y,z,c)$,  by the Fenchel-Moreau theorem, we have
\begin{equation*}
    f(t,X_t,y,z,c)=\sup_{(\alpha,\beta,\gamma)\in D(t,\omega)}\Big\{
    \alpha y+\beta z+\gamma c - f^*(t,X_t,\alpha,\beta,\gamma)\Big\},
\end{equation*}
where $D(t,\omega)=\{(\alpha,\beta,\gamma)\in\mathbb{R}\times\mathbb{R}^d\times\mathbb{R}: f^*(t,X_t,\alpha,\beta,\gamma)<\infty \}$ is the effective domain/barrier cone of $f^*$. For $(\alpha,\beta, \gamma)\in D(t,\omega)$, we further introduce
\begin{equation*}
    f^{\alpha,\beta,\gamma}(t,X_t,y,z,c)=\alpha y+\beta z+\gamma c - f^*(t,X_t,\alpha,\beta,\gamma).
\end{equation*}
To ensure the validity of the comparison theorem for an auxiliary BSDEs introduced later on, we make an additional assumption:
\begin{assumption}\label{(H4)}
The effective domain $D(t,\omega)$ is a compact set in $\mathbb{R}\times\mathbb{R}^d\times[-\lambda,\infty)$.
\end{assumption}

We establish the following main result of the paper, and will prove it in the rest of the paper.
\begin{theorem}\label{result-zc-nonlinear}
Suppose that \cref{(H1),(H2),(H3)} hold and the generator $f(t,x,y,z,c)$ of BSDE  \eqref{FBSDE-Stopping} is convex in $(y,z,c)$, with the effective domain of its convex dual satisfying Assumption \ref{(H4)}. Then,  the value function for \cref{(R-OSP-P)} is given by
\begin{align}\label{eqresult-zc-nonlinear}
V^\lambda_t = \hat{Y}_t^\lambda = \sup_{\tau\in \mathcal{R}_{T_i}(\lambda)} \tilde{Y}_t^\tau , \quad \text{conditional on } \{T_{i-1} \leq t < T_i\}, \quad i \geq 1,
\end{align}
where $\hat{Y}^\lambda_t$ and $\tilde{Y}^\tau_t$ are the unique solutions of \eqref{FBSDE-Penalization} and \eqref{FBSDE-Stopping}, respectively. Moreover, the optimal stopping time is given by
\begin{align}\label{Opti-Stop}
\hat{\tau}^*_{t} := \inf \Big\{T_n >t \,\Big|\, \hat{Y}^\lambda_{T_n} \leq g(T_n, {X}_{T_n})\Big\} \wedge T.
\end{align}
\end{theorem}

\begin{remark}\label{finite-terms}
By partitioning the sample space $\Omega$ using the Poisson stopping times $\{T_i\}_{i\geq 0}$, we can deduce that
\begin{align*}
\hat{Y}_t^\lambda = \sum_{i \geq 1} \mathbb{I}_t^i \hat{Y}_t^\lambda = \sum_{i \geq 1} \mathbb{I}_t^i \sup_{\tau\in \mathcal{R}_{T_i}(\lambda)} \tilde{Y}_t^\tau ,
\end{align*}
Furthermore, the representation \eqref{eqresult-zc-nonlinear} can be expressed more concisely using $\mathcal{R}_{t+}(\lambda)$:
\begin{align*}
\hat{Y}_t^\lambda = \sup_{\tau\in \mathcal{R}_{t+}(\lambda)} \tilde{Y}_t^\tau , \quad t\in[0,T).
\end{align*}
\end{remark}

\begin{remark}\label{degenerate-A0-nozc}
If we make additional assumptions that $f$ satisfies the conditions stated in \cref{recursive-nonre}, Theorem \ref{result-zc-nonlinear} yields the following result, consistent with \cite[Theorem 1.2]{liang2015stochastic}:
\begin{align*}
{Y}_t^{\lambda} = \sup_{\tau\in \mathcal{R}_{T_i}(\lambda)} \tilde{Y}_t^\tau , \quad \text{conditional on } \{T_{i-1} \leq t < T_i\}, \quad i \geq 1,
\end{align*}
where $({Y}^{\lambda}, {Z}^{\lambda})$ is the unique solution of \eqref{FBSDE-Penalization-2}.
\end{remark}

\section{Proof of Theorem \ref{result-zc-nonlinear}}
{\gl We prove Theorem \ref{result-zc-nonlinear} in three steps, each building on the previous one. We begin with Step 1, where the generator \( f \) is independent of the hedging processes \( Z \) and \( C \). We then extend the analysis to Step 2, the linear case, and finally to Step 3, the convex case. Each step requires different techniques.
}

\subsection{Representation for the case with generator independent of \texorpdfstring{{\boldmath$Z$}}{Z} and \texorpdfstring{{\boldmath$C$}}{C}}
\label{sec:repre-noZC}

In this section, we study \cref{(R-OSP-P)} when the generator of the recursive objective functional  \eqref{FBSDE-Stopping} is independent of the hedging processes $Z$ and $C$, i.e. $f(t,x,y,z,c) = f(t,x,y)$.
In such a case,
PBSDE \eqref{FBSDE-Penalization}  simplifies to the following form:
\begin{align}\label{FBSDE-Penalization-nozc}
\begin{aligned}
\hat{Y}_t^\lambda &= g({ T}, {X}_{ T}) + \int_t^{T} \Big[f(s, {X}_s, \hat{Y}_s^\lambda)+\lambda \left( g(s,{X}_s) + {A}_s - \hat{Y}_{s}^\lambda - \hat{C}_{s}^\lambda \right)^+\Big] ds\\
&\qquad- \int_t^{ T} \hat{Z}_s^\lambda dW_s - \int_{t}^T \hat{C}_s^\lambda \tilde{N}(ds),\quad t\in[0, T].
\end{aligned}
\end{align}

The proof of the independent case relies on two key lemmas.

\begin{lemma}\label{lem-recursive-equation}
Suppose that \cref{(H1),(H2),(H3)} hold. Then, the process $\hat{Y}^\lambda$ given by \eqref{FBSDE-Penalization-nozc} is the unique solution of the recursive equation
\begin{align}\label{recursive-equation}
\begin{aligned}
\hat{Y}^\lambda_{t} &= \sum_{i \geq 1} \mathbb{I}_t^i \cdot \mathbb{E}\Bigg[ \int_{t}^{T_i\wedge T} f(s, {X}_s, \hat{Y}^\lambda_s)ds + g(T, {X}_T)\mathbb{I}_{\{T_i > T\}}\\
&\qquad+ \max\Big\{  g(T_i, {X}_{T_i}) , \hat{Y}^\lambda_{T_i} \Big\}\mathbb{I}_{\{T_i\leq T\}} \Bigg| \mathcal{G}_{t}\Bigg].
\end{aligned}
\end{align}
\end{lemma}
\begin{proof}
By noting that $\lambda\Big( g(t,X_t) + A_t - y -c \Big)^+ = \lambda\max\Big\{ g(t,X_t) + A_t, y + c \Big\} - \lambda \big(y + c\big) $, and applying the Jacod-Pham decomposition to the generator of PBSDE (\ref{FBSDE-Penalization-nozc}), we obtain
\begin{align*}
\begin{aligned}
f(t,{X}_t, y) &= \sum_{i \geq 0} f^{i,d}(t, y,T_{(i)}) \mathbb{I}_t^{i+1},\\
\lambda\max\Big\{ g(t,{X}_t) + {A}_t, y + c \Big\}
&= \sum_{i \geq 0} \lambda \max\Big\{\big( g^i(t, T_{(i)}) + {A}^i(t, T_{(i)}), y + c\Big\} \mathbb{I}_t^{i+1},
\end{aligned}
\end{align*}
where the process $f^{i,d}(\cdot, \theta_{(i)})\in L_{\mathbb{F}}^2(0,T;\mathbb{R})$ for any $\theta_{(i)} \in \Delta_i$.
Hence, we obtain the decomposition according to \cref{Decompo-Y-lambda}: for any $t\in[0,T]$,
\begin{align}\label{decompo-mainresult-Y}
\begin{aligned}
\hat{Y}^\lambda_t &= \sum_{i \geq 0} \hat{Y}^{i,d}_t (T_{(i)}) \mathbb{I}_t^{i+1},\\
\hat{Z}^\lambda_t &= \hat{Z}^{0,d}_t \mathbb{I}_{\{t \leq T_1\}} + \sum_{i \geq 1} \hat{Z}^{i,d}_t (T_{(i)}) \mathbb{I}_{\{T_i < t \leq T_{i+1}\}},\\
\hat{C}^\lambda_t &= \hat{C}^{0,d}_t \mathbb{I}_{\{t \leq T_1\}} + \sum_{i \geq 1} \hat{C}^{i,d}_t (T_{(i)}) \mathbb{I}_{\{T_i < t \leq T_{i+1}\}},
\end{aligned}
\end{align}
with the processes
\begin{align*}
\hat{Y}^{i,d}_t (T_{(i)}) &:= \hat{Y}^{i,d}_t (\theta_{(i)})|_{\theta_{(i)} = T_{(i)}},\\
\hat{Z}^{i,d}_t (T_{(i)}) &:= \hat{Z}^{i,d}_t (\theta_{(i)})|_{\theta_{(i)} = T_{(i)}},\\
\hat{C}^{i,d}_t (T_{(i)}) &:= \hat{C}^{i,d}_t (\theta_{(i)})|_{\theta_{(i)} = T_{(i)}},
\end{align*}
given by the following Brownian motion driven BSDE in the time horizon $[\theta_i\wedge T, T]$ with $ \hat{C}_s^{i,d}(\theta_{(i)}) := \hat{Y}_s^{i+1,d}((\theta_{(i)}, s)) - \hat{Y}_s^{i,d}(\theta_{(i)})$ for $i\geq 0$:
\begin{align*}
\begin{aligned}
\hat{Y}_{t}^{i,d} (\theta_{(i)}) &= g^i(T, \theta_{(i)}) + \int_{t}^{T} \Big[f^{i,d}(s, \hat{Y}_s^{i,d}(\theta_{(i)}), \theta_{(i)}) + \lambda \Big( g^i(s, \theta_{(i)})+ {A}^i(s, \theta_{(i)})\\
&\qquad - \hat{Y}_s^{i,d}(\theta_{(i)}) - \hat{C}_s^{i,d}(\theta_{(i)}) \Big)^+ + \lambda\hat{C}_s^{i,d}(\theta_{(i)}) \Big] ds - \int_{t}^{T} \hat{Z}_s^{i,\lambda}(\theta_{(i)}) dW_s.
\end{aligned}
\end{align*}

Applying It\^o's formula to $e^{-\lambda t}\hat{Y}_{t}^{i,d} (\theta_{(i)})$, we obtain
\begin{align*}
\hat{Y}_{t}^{i,d} (\theta_{(i)}) &= e^{-\lambda (T-t)}g^i(T, \theta_{(i)})
+ \int_{t}^{T} e^{-\lambda (s-t)}
\Big[ f^{i,d}(s, \hat{Y}_s^{i,d}(\theta_{(i)}), \theta_{(i)})  \\
&\quad + \lambda \max \Big\{ g^i(s, \theta_{(i)})
+ {A}^i(s, \theta_{(i)}), \hat{Y}_s^{i,d}(\theta_{(i)})
+ \hat{C}_s^{i,d}(\theta_{(i)}) \Big\} \Big] ds \\
&\quad - \int_{t}^{T} e^{-\lambda (s-t)} \hat{Z}_s^{i,\lambda}(\theta_{(i)}) dW_s.
\end{align*}
Taking conditional expectation with respect to $\mathcal{F}_t$ on both sides and substituting it into \eqref{decompo-mainresult-Y}, we obtain
\begin{align}\label{decomp-Ylambda-exp}
\begin{aligned}
\hat{Y}^\lambda_t
&= \sum_{i \geq 1} \mathbb{I}_t^{i} \cdot
\mathbb{E} \Bigg[ e^{-\lambda (T-t)} g^{i-1}(T, \theta_{(i-1)}) \\
&\quad + \int_{t}^{T} e^{-\lambda (s-t)}
\Big[ f^{i-1,d}(s,  \hat{Y}_s^{i-1,d}(\theta_{(i-1)}), \theta_{(i-1)}) \\
&\quad + \lambda \max \Big\{ g^{i}(s, (\theta_{(i-1)},s)),
\hat{Y}_s^{i,d}((\theta_{(i-1)},s)) \Big\} \Big] ds
\Bigg| \mathcal{F}_t \Bigg] \Bigg|_{\theta_{(i-1)} = T_{(i-1)}}.
\end{aligned}
\end{align}

{\gl To verify that $\hat{Y}^{\lambda}$ satisfies the recursive equation \eqref{recursive-equation}, we further show that each term on the RHS of \eqref{decomp-Ylambda-exp} is equal to the RHS of \eqref{recursive-equation} on the set $\mathbb{I}_t^i$, $i\geq 1$.

Note that on the set $\mathbb{I}_t^i\cap\mathbb{I}_{\{T_i>T\}}$, we have $g(T,X_T)=g^{i-1}(T,T_{(i-1)})$ by the Jacod-Pham decomposition, and $T_{(i-1)}$ is $\mathcal{G}_t$-measurable. Hence,
\begin{align}\label{Poisson-eq0-3}
\begin{aligned}
\mathbb{I}_t^i \cdot \mathbb{E}\Bigg[ g(T,{X}_T) \mathbb{I}_{\{T_i>T\}} \Bigg| \mathcal{G}_t \Bigg]
&= \mathbb{I}_t^i \cdot \mathbb{E}\Bigg[ g^{i-1}(T, \theta_{(i-1)}) \mathbb{I}_{\{T_i>T\}}  \Bigg| \mathcal{G}_t \Bigg] \Bigg|_{\theta_{(i-1)} = T_{(i-1)}}\\
&= \mathbb{I}_t^i \cdot \mathbb{E}\Bigg[ e^{-\lambda (T-t)} g^{i-1}(T, \theta_{(i-1)})   \Bigg| \mathcal{F}_t \Bigg] \Bigg|_{\theta_{(i-1)} = T_{(i-1)}},
\end{aligned}
\end{align}
where we used the filtration switching formula from $\mathcal{G}_t$ to $\mathcal{F}_t$ (see \cite[Lemma 2.1]{liang2015stochastic}) in the last equality. Moreover, by the Jacod-Pham decomposition, we have
\[
\max\Big\{ g(T_i,{X}_{T_i}), \hat{Y}_{T_i}^\lambda \Big\}=\max\Big\{ g^{i}( T_i, T_{(i)}),\hat{Y}_{T_i}^{i,d}(T_{(i)}) \Big\}.
\]
Since $T_{(i-1)}$ is $\mathcal{G}_t$-measurable, it follows that
\begin{align}\label{Poisson-eq0-2}
\begin{aligned}
&\mathbb{I}_t^i \cdot \mathbb{E}\Bigg[ \max\Big\{ g(T_i,{X}_{T_i}), \hat{Y}_{T_i}^\lambda \Big\} \mathbb{I}_{\{T_i \leq T\}} \Bigg| \mathcal{G}_t \Bigg] \\
&= \mathbb{I}_t^i \cdot \mathbb{E}\Bigg[ \max\Big\{ g^{i}( T_i, (\theta_{(i-1)}, T_i)),\hat{Y}_{T_i}^{i,d}((\theta_{(i-1)}, T_i)) \Big\}  \mathbb{I}_{\{T_i \leq T\}}\Bigg| \mathcal{G}_t \Bigg] \Bigg|_{\theta_{(i-1)} = T_{(i-1)}}\\
&= \mathbb{I}_t^i \cdot \mathbb{E}\Bigg[\int_t^T \lambda e^{-\lambda (s-t)} \max\Big\{ g^{i}( s, (\theta_{(i-1)},s)),\hat{Y}_{s}^{i,d}((\theta_{(i-1)},s)) \Big\}  ds  \Bigg| \mathcal{F}_t \Bigg] \Bigg|_{\theta_{(i-1)} = T_{(i-1)}},
\end{aligned}
\end{align}
where we applied the filtration switching formula from $\mathcal{G}_t$ to $\mathcal{F}_t$ in the last equality. Similarly, we obtain
\begin{align}\label{Poisson-eq2}
\begin{aligned}
&\mathbb{I}_t^i \cdot \mathbb{E}\Bigg[\int_t^{T_i\wedge T} f(s,{X}_s, \hat{Y}^\lambda_s)ds \Bigg| \mathcal{G}_t \Bigg] \\
&= \mathbb{I}_t^i \cdot \mathbb{E}\Bigg[ e^{-\lambda (T-t)} \int_t^{T} f^{i-1,d}(s,  \hat{Y}_s^{i-1,d}(\theta_{(i-1)}), \theta_{(i-1)}) ds \\
&\quad + \int_t^{T} \lambda e^{-\lambda (r-t)} \int_t^{r} f^{i-1,d}(s,  \hat{Y}_s^{i-1,d}(\theta_{(i-1)}), \theta_{(i-1)}) ds dr  \Bigg| \mathcal{F}_t \Bigg] \Bigg|_{\theta_{(i-1)} = T_{(i-1)}}\\
&= \mathbb{I}_t^i \cdot \mathbb{E}\Bigg[\int_t^{ T} e^{-\lambda (s-t)} f^{i-1,d}(s,  \hat{Y}_s^{i-1,d}(\theta_{(i-1)}), \theta_{(i-1)}) ds \Bigg| \mathcal{F}_t \Bigg] \Bigg|_{\theta_{(i-1)} = T_{(i-1)}}.
\end{aligned}
\end{align}
where we used integration by parts on the term
\[
e^{-\lambda (r-t)} \cdot \int_t^r f^{i-1,d}(s,  \hat{Y}_s^{i-1,d}(\theta_{(i-1)}), \theta_{(i-1)}) ds
\]
with respect to \( r \) from \( t \) to \( T \) in the last equality.

Finally, substituting \eqref{Poisson-eq0-3}, \eqref{Poisson-eq0-2}, and \eqref{Poisson-eq2} into \eqref{decomp-Ylambda-exp}, we obtain \eqref{recursive-equation}, which completes the proof.}
\end{proof}

Next, we introduce an auxiliary process in the time interval $[0,T]$:
\begin{align}\label{Auxiliary-process}
\hat{Q}^\lambda_s := \left\{
\begin{aligned}
&\max\Big\{ g(s ,{X}_{s}), \hat{Y}_{s}^\lambda \Big\}, \quad & s = T_i, \quad i \geq 0,\\
&\hat{Y}_{s}^\lambda, \quad & otherwise.
\end{aligned}
\right.
\end{align}
Then, by definition and the recursive equation \eqref{recursive-equation}, we have the dynamic programming equation
\begin{align}\label{recursive-equation-2}
\begin{aligned}
\hat{Q}^\lambda_{T_{i-1}\wedge T}
&= \max \Bigg\{ g(T_{i-1}\wedge T, {X}_{T_{i-1}\wedge T}),
\hat{Y}^{\lambda}_{T_{i-1}\wedge T} \Bigg\} \\
&= \max \Bigg\{ g(T_{i-1}\wedge T, {X}_{T_{i-1}\wedge T}), \\
&\quad \mathbb{E} \Bigg[ \int_{T_{i-1}\wedge T}^{T_i\wedge T}
f(s, {X}_s, \hat{Q}^\lambda_s) ds
+ \hat{Q}^{\lambda}_{{T_i}\wedge T}
\Bigg| \mathcal{G}_{T_{i-1}\wedge T} \Bigg] \Bigg\},
\end{aligned}
\end{align}
for any $i \geq 1$.
\begin{lemma}\label{stopping-recursive equation}
Suppose that \cref{(H1),(H2),(H3)} hold. Consider the following auxiliary optimal stopping problem
\begin{align*}
R^\lambda_t :=\sup_{\tau\in \mathcal{R}_{t}(\lambda)}  \tilde{Y}_{t}^{\tau}=\tilde{Y}_t^{{\tau}^*_t},
\end{align*}
with the optimal stopping time denoted as ${\tau}_t^*$. Then, its value function ${R}_t^\lambda$ is given by
\begin{align}\label{recursive-process}
\begin{aligned}
{R}^\lambda_t &=\hat{Q}^\lambda_t= \sup_{\tau\in \mathcal{R}_{t}(\lambda)}  \tilde{Y}_{t}^{\tau}  , \quad t\in[0,T],
\end{aligned}
\end{align}
where $\hat{Q}^\lambda$ and $\tilde{Y}^\tau$ are the unique solution of \eqref{Auxiliary-process} and \eqref{FBSDE-Stopping}, respectively. Moreover, the optimal stopping time is given by
\begin{align}\label{Opti-Stop-auxiliary}
\tau^*_t := \inf \Big\{T_n\geq t \,\Big|\, \hat{Q}^\lambda_{T_n} = g(T_n, {X}_{T_n})\Big\} \wedge T.
\end{align}
\end{lemma}
\begin{proof}
The dynamic programming equations \eqref{recursive-equation-2} implies that
\begin{align*}
\begin{aligned}
\hat{Q}^\lambda_{T_{i}\wedge T} &= \hat{Y}_{T_{i}\wedge T}^\lambda + \Big(\hat{Q}^\lambda_{T_{i}\wedge T} - \hat{Y}_{T_{i}\wedge T}^\lambda\Big)\\
&= \mathbb{E}\Big[ \hat{Q}^\lambda_{T_{i+1}\wedge T} + \int_{T_{i}\wedge T}^{T_{i+1}\wedge T} f(s, {X}_s, \hat{Q}^\lambda_s)ds\Big| \mathcal{G}_{T_{i}\wedge T} \Big] + \Big(\hat{Q}^\lambda_{T_{i}\wedge T} - \hat{Y}_{T_{i}\wedge T}^\lambda\Big).
\end{aligned}
\end{align*}
In turn, for any $\tau\in \mathcal{R}_{t}(\lambda)$, we have
\begin{align*}
\begin{aligned}
\hat{Q}^\lambda_{t} &= \hat{Y}_{t}^\lambda + \Big(\hat{Q}^\lambda_{t} - \hat{Y}_{t}^\lambda\Big)\\
&= \sum_{i \geq 1} \mathbb{I}_t^i \cdot \mathbb{E}\Big[ \hat{Q}^\lambda_{T_i\wedge T} + \int_{t}^{T_i\wedge T} f(s, {X}_s, \hat{Q}^\lambda_s)ds\Big| \mathcal{G}_{t} \Big] + \Big(\hat{Q}^\lambda_{t} - \hat{Y}_{t}^\lambda\Big)\\
&= \mathbb{E}\Big[ \hat{Q}^\lambda_{\tau} + \int_{t}^{\tau} f(s, {X}_s, \hat{Q}^\lambda_s)ds\Big| \mathcal{G}_{t} \Big]+ \mathbb{E}\Big[\sum_{t\leq T_{j-1}< \tau} \Big(\hat{Q}^\lambda_{T_{j-1}\wedge T} - \hat{Y}_{T_{j-1}\wedge T}^\lambda\Big)\Big|\mathcal{G}_t\Big].
\end{aligned}
\end{align*}
By the martingale representation theorem, there exist a pair of adapted processes \((\hat{Z}^\tau_s, \hat{C}^\tau_s) \in L_\mathbb{G}^{2,p}(t, T; \mathbb{R}^d) \times L_\mathbb{G}^{2,p}(t, T; \mathbb{R})\) such that
\begin{align}\label{Q-DiscreteForm}
\begin{aligned}
\hat{Q}^\lambda_{t} &= \hat{Q}^\lambda_{\tau} + \int_{t}^{\tau} f(s, {X}_s, \hat{Q}^\lambda_s)ds - \int_{t}^{\tau} \hat{Z}^\tau_sdW_s - \int_{t}^{\tau} \hat{C}^\tau_t\tilde{N}(ds) \\
&\qquad + \mathbb{E}\Big[\sum_{t\leq T_{j-1}< \tau} \Big(\hat{Q}^\lambda_{T_{j-1}\wedge T} - \hat{Y}_{T_{j-1}\wedge T}^\lambda\Big)\Big|\mathcal{G}_t\Big].
\end{aligned}
\end{align}

We compare BSDEs \eqref{FBSDE-Stopping} and \eqref{Q-DiscreteForm}.
Since $\hat{Q}^\lambda_{T_{j-1}\wedge T} - \hat{Y}_{T_{j-1}\wedge T}^\lambda\geq 0$ and $\hat{Q}^\lambda_{\tau}\geq g(\tau, {X}_{\tau})$, we have $\hat{Q}^\lambda_{t} \geq \tilde{Y}_{t}^\tau$, $a.s.$ for any $\tau \in \mathcal{R}_{t}(\lambda)$ by the comparison theorem of BSDEs with jumps \cite[Theorem 4.4]{quenez2013bsdes}. Taking the supremum over $\tau\in \mathcal{R}_{t}(\lambda)$, we obtain $\hat{Q}^\lambda_{t} \geq \sup_{\tau\in \mathcal{R}_{t}(\lambda)} \tilde{Y}_{t}^{\tau}$.

Finally, we prove the equality and show that the stopping time $\tau^*_t$ defined in \eqref{Opti-Stop-auxiliary} is indeed the optimal stopping time of the auxiliary optimal stopping problem. To this end, it is sufficient to show that $\hat{Q}^\lambda_{t} = \tilde{Y}_{t}^{\tau^*_t}$. By \eqref{Auxiliary-process}, we obtain $\hat{Q}^\lambda_{T_{j-1}\wedge T} - \hat{Y}_{T_{j-1}\wedge T}^\lambda= 0$ for $t \leq T_{j-1} <\tau^*_t$ and $\hat{Q}^\lambda_{\tau^*_t} = g(\tau^*_t, {X}_{\tau^*_t})$.  Hence, the uniqueness of the solution for BSDEs with jumps \cite[Theorem 2.1]{wu2003fully} implies that $\hat{Q}^\lambda_{t} = \tilde{Y}_{t}^{\tau^*_t}$, $a.s.$.
\end{proof}
\begin{remark}
This lemma requires that the generator $f$ in the recursive objective functional does not depend on the variables $Z$ and $C$. Otherwise, it is unclear how to represent the variables from the perspective of the optimal stopping problem, and then one cannot connect the pair $(\tilde{Z}^\tau_t,\tilde{C}^\tau_t)$ in BSDE \cref{{FBSDE-Stopping}} and the pair $(\hat{Z}^\lambda_t,\hat{C}^\lambda_t)$ in PBSDE \eqref{FBSDE-Penalization}.
\end{remark}

Based on the above two lemmas, we are ready to complete the proof of the independent case.
\begin{proof}[Proof of \cref{result-zc-nonlinear} — Step 1: Case $f(t,x,y,z,c) = f(t,x,y)$]
\noindent\\
We have $\hat{Q}^\lambda_t = \hat{R}^\lambda_t$ according to \cref{stopping-recursive equation}. Since the process $\tilde{Y}_{t}^{\tau}$ is c\`adl\`ag, it follows that
\begin{align*}
\hat{Q}^\lambda_{t+} = \hat{R}^\lambda_{t+} = \sup_{\tau\in \mathcal{R}_{t+}(\lambda)} \tilde{Y}_{t+}^{\tau} = \sup_{\tau\in \mathcal{R}_{t+}(\lambda)} \tilde{Y}_{t}^{\tau} = V^\lambda_t, \quad t\in[0,T).
\end{align*}
Since the process $\hat{Y}^\lambda$ is also c\`adl\`ag, we apply the recursive equation \eqref{Auxiliary-process} to obtain
\begin{align*}
V^\lambda_t = \hat{Q}^\lambda_{t+} = \hat{Y}^\lambda_{t+} = \hat{Y}^\lambda_{t}, \quad t\in[0,T).
\end{align*}
Moreover, the optimal stopping time $\hat{\tau}^*_t$ satisfies
\begin{align*}
\begin{cases}
\hat{\tau}^*_t = \tau^*_t, \quad &t\in(T_{i-1},T_i), \\
\hat{\tau}^*_{T_{i-1}} = \tau^*_{T_i}, \quad &t=T_{i-1},
\end{cases} \quad i \geq 1.
\end{align*}
Thus, equation \eqref{Opti-Stop} follows from the recursive equation \eqref{Auxiliary-process}.
\end{proof}

\subsection{Representation for the case with generator linear with respect to \texorpdfstring{{\boldmath$Z$}}{Z} and \texorpdfstring{{\boldmath$C$}}{C}}
\label{sec:repre-ZC}

This section extends the independent case to a more general setting where the generator is nonlinear with respect to $X$ and $Y$ but linear with respect to $Z$ and $C$ with the help of an adjoint process. Suppose that $f(t,x,y,z,c) = \rho(t,x,y) + \beta_t z + \gamma_t \lambda c$ where $\gamma_t > - 1$ and $\beta,\,\gamma$ are bounded {$\mathbb{G}$-predictable processes}.
In such a case, the generator of
PBSDE \eqref{FBSDE-Penalization} simplifies to the following form
\begin{align*}
\begin{aligned}
&f(t,{X}_{t},y,z,c+(g(t,X_t)+A_t-y-c)^+) +\lambda \Big( g(t,{X}_t) + A_t - y-c\Big)^+ \\
&= \rho(t,X_t,y)+\beta_tz+\gamma_t\lambda c+(1+\gamma _t)\lambda \Big( g(t,{X}_t) + A_t - y-c\Big)^+,
\end{aligned}
\end{align*}
and therefore PBSDE \eqref{FBSDE-Penalization} simplifies to
\begin{align}\label{FBSDE-Penalization-zc}
\begin{aligned}
%
\hat{Y}_t^\lambda &= g({ T}, {X}_T) + \int_t^{T}\Big[\rho(s, {X}_s, \hat{Y}_s^\lambda) + \beta_s \hat{Z}_s^\lambda + \gamma_s \lambda \hat{C}_s^\lambda\\
&\qquad +\big(1+ \gamma_s\big) \lambda \big(g(s,X_s) + {A}_s -\hat{Y}_{s}^\lambda - \hat{C}_{s}^\lambda)^+\Big]  ds \\
&\qquad - \int_t^{ T} \hat{Z}_s^\lambda dW_s - \int_t^{ T}  \hat{C}_s^\lambda  \tilde{N}(ds),\quad t\in[0, T].
\end{aligned}
\end{align}
It is clear that when $\beta_t = \gamma_t \equiv 0$, the penalized equation \eqref{FBSDE-Penalization-zc} will degenerate to \eqref{FBSDE-Penalization-nozc}.

To prove the linear case,
we first introduce an adjoint process $\Gamma^{\beta,\gamma}$: for any given bounded and $\mathbb{G}$-predictable $(\beta,\,\gamma)$,
\begin{align*}
\left\{
\begin{aligned}
d\Gamma_{t}^{\beta,\gamma} &= \Gamma_{t-}^{\beta,\gamma} \Big(\beta_t dW_t + {\gamma_t} \tilde{N}(dt) \Big), \quad t\in[0,T],\\
\Gamma_{0}^{\beta,\gamma} &= 1.
\end{aligned}
\right.
\end{align*}
According to the Jacod-Pham decomposition, we have
\begin{equation*}
\Gamma_t^{\beta,\gamma}= \sum_{i \geq 0} \Gamma^{i,\beta,\gamma}_t(T_{(i)}) \mathbb{I}_t^{i+1},\quad {\gamma_t=\gamma_t^{0}\mathbb{I}_{\{t\leq T_{1}\}} + \sum_{i\geq 1}\gamma_t^{i}(T_{(i)})\mathbb{I}_{\{T_i<t\leq T_{i+1}\}}} 
\end{equation*}
where $\Gamma^{i,\beta,\gamma}(\theta_{(i)}), \gamma^{i}(\theta_{(i)})\in L_{\mathbb{F}}^2(0,T;\mathbb{R})$ for any $\theta_{(i)}\in \Delta_i$. Moreover, it holds that
\begin{equation}\label{auxilary_relation}
\Gamma_t^{i+1,\beta,\gamma}((T_{(i)},t))=\Gamma_{t}^{i,\beta,\gamma}(T_{(i)})(\gamma_t^{i}(T_{(i)})+1),\quad t\in[T_{i},T_{i+1}].
\end{equation}

Applying It\^{o}'s formula to $\Gamma^{\beta,\gamma}\tilde{Y}^{\tau}$ yields that
\begin{align*}
\begin{aligned}
\Gamma_{t}^{\beta,\gamma} \tilde{Y}_t^\tau &= \Gamma_{\tau}^{\beta,\gamma} \tilde{Y}_{\tau}^\tau + \int_{t\wedge\tau}^{\tau} \Gamma_{r}^{\beta,\gamma} \rho(r,{X}_r,\tilde{Y}_{r}^\tau) dr - \int_{t\wedge\tau}^{\tau} \Gamma_{r}^{\beta,\gamma}\Big( \tilde{Z}_{r}^\tau + \beta_r \tilde{Y}_{r}^\tau\Big) dW_r\\
&\qquad - \int_{t\wedge\tau}^{\tau } \Gamma_{r-}^{\beta,\gamma} \Big[ (1 + \gamma_r) \tilde{C}_{r}^\tau + \gamma_r \tilde{Y}_{r-}^\tau\Big] \tilde{N}(dr), \quad t\in[0,T].
\end{aligned}
\end{align*}
Hence, the triplet
\[
\left\{
\begin{aligned}
\tilde{y}_t^\tau &:= \Gamma_{t}^{\beta,\gamma} \tilde{Y}_{t}^\tau, \\
\tilde{z}_t^\tau &:= \Gamma_{t}^{\beta,\gamma} \left( \tilde{Z}_{t}^\tau + \beta_t \tilde{Y}_{t}^\tau \right), \\
\tilde{c}_t^\tau &:= \Gamma_{t-}^{\beta,\gamma} \left[ (1 + \gamma_t) \tilde{C}_{t}^\tau + \gamma_t \tilde{Y}_{t-}^\tau \right]
\end{aligned}
\right.
\]
satisfies
\begin{align*}
\begin{aligned}
\tilde{y}_{t}^\tau &= \Gamma_{\tau}^{\beta,\gamma}g({\tau}, {{X}_{\tau}}) + \int_{t\wedge\tau}^{\tau} \Gamma_{r}^{\beta,\gamma} \rho(r,{X}_r, (\Gamma_{r}^{\beta,\gamma})^{-1} \tilde{y}_{r}^\tau)dr \\
&\qquad - \int_{t\wedge\tau}^{\tau} {\tilde{z}_r^\tau} dW_r - \int_{t\wedge\tau}^{\tau} \tilde{c}^\tau_r \tilde{N}(dr),\quad t\in[0,T].
\end{aligned}
\end{align*}
The new generator $ f(t,x,y): = \Gamma_{t}^{\beta,\gamma} \rho(t,x,(\Gamma_{t}^{\beta,\gamma})^{-1} y)$ is $\mathbb{G}$-progressively measurable and independent with the variables $z$ and $c$. By \cref{FBSDE-Penalization-nozc} for the independent case,
the value function of the auxiliary problem relevant to the recursive objective functional $\tilde{y}^\tau$, i.e.
\begin{align}\label{eqresult-1}
\begin{aligned}
\hat{y}_{t}^\lambda&= \sup_{\tau\in\mathcal{R}_{t+}(\lambda)} \tilde{y}_{t}^\tau ,
\end{aligned}
\end{align}
is given by
\begin{align}\label{y-lambda}
\begin{aligned}
\hat{y}_t^\lambda &= \Gamma_{T}^{\beta,\gamma}g({ T}, {{X}_T}) + \int_t^T \Big[ \Gamma_{r}^{\beta,\gamma} \rho(r,{X}_r, (\Gamma_{r}^{\beta,\gamma})^{-1} \hat{y}^\lambda_r) + \lambda \Big( \Gamma_{r}^{\beta,\gamma} g(r,{X}_r) + A_r^\Gamma\\
&\qquad - \hat{y}_r^\lambda - \hat{c}_r^\lambda \Big)^+ \Big] dr - \int_t^T \hat{z}^\lambda_r dW_r- \int_t^T \hat{c}_r^\lambda \tilde{N}(dr),
\end{aligned}
\end{align}
and the optimal stopping time is given by
\begin{align}\label{optimalstopping-y}
\begin{aligned}
\hat{\tau}^{*,y}_t :=& \,\inf \Big\{T_n >t \,\Big|\, \hat{y}^\lambda_{T_n} \leq \Gamma_{T_n}^{\beta,\gamma} g(T_n, {X}_{T_n})\Big\} \wedge T.
\end{aligned}
\end{align}
Herein, the auxiliary process $A^\Gamma \in L_{\mathbb{G}}^{2}(0,T;\mathbb{R})$ in PBSDE \eqref{y-lambda} admits the form
\begin{align*}
\begin{aligned}
{A}_r^\Gamma &:=\sum_{i\geq 0} A^{i,\Gamma}(r, T_{(i)}) \mathbb{I}_r^{i+1},\\
\end{aligned}
\end{align*}
where the process $A^{i,\Gamma}(r, T_{(i)}) := A^{i,\Gamma}(r, \theta_{(i)})|_{\theta_{(i)} = T_{(i)}}$ is given by
$$ A^{i,\Gamma}(r, \theta_{(i)}) := \Gamma_r^{i+1,\beta,\gamma}((\theta_{(i)},r))g^{i+1}(r, (\theta_{(i)}, r)) - \Gamma_r^{i,\beta,\gamma}(\theta_{(i)})g^{i}(r, \theta_{(i)})$$ for $i \geq 0$.

Next, by It\^o's formula, we have
\begin{align*}
\left\{
\begin{aligned}
d(\Gamma_{t}^{\beta,\gamma})^{-1} &= (\Gamma_{t-}^{\beta,\gamma})^{-1} \Big(-\beta_t (dW_t-\beta_tdt) - \frac{\gamma_t}{1+\gamma_t} (\tilde{N}(dt)-\gamma_t\lambda dt) \Big), \quad t\in[0,T],\\
(\Gamma_{0}^{\beta,\gamma})^{-1} &= 1.
\end{aligned}
\right.
\end{align*}
In turn, applying It\^o's formula to $(\Gamma^{\beta,\gamma})^{-1} \hat{y}^\lambda$ yields that
\begin{align*}
\begin{aligned}
(\Gamma_{t}^{\beta,\gamma})^{-1} \hat{y}^\lambda_t &= (\Gamma_{T}^{\beta,\gamma})^{-1} \hat{y}^\lambda_T + \int_t^T \Big[ \rho(r,{X}_r,(\Gamma_{r}^{\beta,\gamma})^{-1} \hat{y}^\lambda_r) + \beta_r(\Gamma_{r}^{\beta,\gamma})^{-1} \Big( \hat{z}^\lambda_r - \beta_r \hat{y}^\lambda_r \Big)\\
&\qquad + \lambda\gamma_r (1 + \gamma_r)^{-1} (\Gamma_{r}^{\beta,\gamma})^{-1} \Big( \hat{c}^\lambda_r - \gamma_r \hat{y}^\lambda_r \Big) + \lambda \Big( g(r,{X}_r) + (\Gamma_{r}^{\beta,\gamma})^{-1} A_r^\Gamma \\
&\qquad - (\Gamma_{r}^{\beta,\gamma})^{-1} \hat{y}_{r}^\lambda - (\Gamma_{r}^{\beta,\gamma})^{-1} \hat{c}_{r}^\lambda \Big)^+ \Big] dr - \int_t^T (\Gamma_{r}^{\beta,\gamma})^{-1} \Big( \hat{z}^\lambda_r - \beta_r \hat{y}^\lambda_r\Big) dW_r \\
&\qquad - \int_t^{T} (1 + \gamma_r)^{-1} (\Gamma_{r-}^{\beta,\gamma})^{-1} \Big( \hat{c}^\lambda_r - \gamma_r \hat{y}^\lambda_{r-}\Big) \tilde{N}(dr).
\end{aligned}
\end{align*}
Then, the triplet
\begin{align}\label{zc-Y-y-Y}
\left\{
\begin{aligned}
\hat{Y}_t^\lambda &:= (\Gamma_{t}^{\beta,\gamma})^{-1} \hat{y}_{t}^\lambda, \\
\hat{Z}_t^\lambda &:= (\Gamma_{t}^{\beta,\gamma})^{-1} \big( \hat{z}^\lambda_t - \beta_t \hat{y}^\lambda_t \big) = (\Gamma_{t}^{\beta,\gamma})^{-1} \hat{z}^\lambda_t - \beta_t \hat{Y}^\lambda_t, \\
\hat{C}_t^\lambda &:= (1 + \gamma_t)^{-1} (\Gamma_{t-}^{\beta,\gamma})^{-1} \big( \hat{c}^\lambda_t - \gamma_t \hat{y}^\lambda_{t-} \big) = (1 + \gamma_t)^{-1} \big( (\Gamma_{t-}^{\beta,\gamma})^{-1} \hat{c}^\lambda_t - \gamma_t \hat{Y}^\lambda_{t-} \big)
\end{aligned}
\right.
\end{align}
satisfies BSDE
\begin{align*}
\begin{aligned}
\hat{Y}_t^\lambda &= g({ T}, {X}_T) + \int_t^{T}\rho(s, {X}_s, \hat{Y}_s^\lambda) + \beta_s \hat{Z}_s^\lambda + \gamma_s \lambda \hat{C}_s^\lambda\\
&\qquad +\lambda \Big( g(s,{X}_s) + (\Gamma_{s}^{\beta,\gamma})^{-1} A_s^\Gamma - \hat{Y}_{s}^\lambda - (\Gamma_{s}^{\beta,\gamma})^{-1} \hat{c}_{s}^\lambda \Big)^+  ds \\
&\qquad - \int_t^{ T} \hat{Z}_s^\lambda dW_s - \int_t^{ T}  \hat{C}_s^\lambda  \tilde{N}(ds).
\end{aligned}
\end{align*}
Note that
{\small
\begin{align*}
\begin{aligned}
(\Gamma_{r}^{\beta,\gamma})^{-1} A_r^\Gamma &= \sum_{i \geq 0} \Big(\Gamma^{i, \beta, \gamma}_r (T_{(i)})\Big)^{-1} A^{i,\Gamma}(r, T_{(i)}) \mathbb{I}_r^{i+1},\\
&= \sum_{i \geq 0} \Big[ \frac{\Gamma^{i+1,\beta,\gamma}_r((T_{(i)}, r))}{\Gamma^{i, \beta, \gamma}_r (T_{(i)})} g^{i+1}(r, (T_{(i)}, r)) - g^{i}(r, T_{(i)}) \Big] \mathbb{I}_r^{i+1}\\
&=\sum_{i \geq 0}  \Big[ (\gamma_r^i(T_{(i)})+1) g^{i+1}(r, (T_{(i)}, r)) - g^{i}(r, T_{(i)}) \Big] \mathbb{I}_r^{i+1}\\
&=\sum_{i \geq 0}  \gamma_r^i(T_{(i)}) g^{i+1}(r, (T_{(i)}, r))\mathbb{I}_r^{i+1} +\sum_{i\geq 0}\Big[g^{i+1}(r, (T_{(i)},r))-g^{i}(r, T_{(i)}) \Big] \mathbb{I}_r^{i+1}\\
&= \gamma_r(g(r,X_r)+A_r) + A_r,
\end{aligned}
\end{align*}
}
where we utilized (\ref{auxilary_relation}) in the third inequality. Moreover, by the definition of $\hat{C}^{\lambda}$, we have
\begin{equation*}
(\Gamma_{r}^{\beta,\gamma})^{-1} \hat{c}_{r}^\lambda=(1+\gamma_r)\hat{C}_r^{\lambda}+\gamma_r\hat{Y}^{\lambda}_r.
\end{equation*}
Hence, the triplet $(\hat{Y}^{\lambda}, \hat{Z}^{\lambda},\hat{C}^{\lambda})$ is the unique soultion of PBSDE (\ref{FBSDE-Penalization}) by its solution uniqueness.

\begin{proof}[Proof of \cref{result-zc-nonlinear} — Step 2: Case $f(t,x,y,z,c) = \rho(t,x,y) + \beta_t z + \gamma_t \lambda c$]
\noindent\\ We verify that \eqref{eqresult-1} implies \eqref{eqresult-zc-nonlinear}, and that $\hat{\tau}^*_t$ is the optimal stopping time. Recall that $\Gamma_{t}^{\beta,\gamma} > 0$ and $\Gamma_{0}^{\beta,\gamma} = 1$. Then, we have
\begin{align*}
\begin{aligned}
\Gamma_{t}^{\beta,\gamma} \hat{Y}_{t}^\lambda &= \hat{y}_{t}^\lambda = \sup_{\tau\in\mathcal{R}_{t+}(\lambda)}  \tilde{y}_{t}^\tau  = \sup_{\tau\in\mathcal{R}_{t+}(\lambda)} \Gamma_{t}^{\beta,\gamma} \tilde{Y}_{t}^\tau  = \Gamma_{t}^{\beta,\gamma} \sup_{\tau\in\mathcal{R}_{t+}(\lambda)} \tilde{Y}_{t}^\tau , \quad t\in[0,T),
\end{aligned}
\end{align*}
and the optimal stopping time $\hat{\tau}^{*,y}_t \in \mathcal{R}_{t+}(\lambda)$ for $\tilde{y}_{t}^\tau$ is also optimal for $\tilde{Y}_{t}^\tau$. Substituting \eqref{zc-Y-y-Y} into \eqref{optimalstopping-y}, we obtain $\hat{\tau}^{*,y}_t = \hat{\tau}^{*}_t$.
\end{proof}

\subsection{Representation for the case with convex/concave generator}
\label{sec:repre-ZC-nonlinear}

In this section, we prove the general case of convex/concave generator.

\begin{proof}[Proof of \cref{{result-zc-nonlinear}} — Step 3: Convex case]
\noindent\\    First, due to the convexity assumption on the generator, it is known that the recursive objective functional admits a convex dual representation (see \cite[section 5]{quenez2013bsdes}): $$\tilde{Y}^{\tau}_t=\sup_{(\alpha,\beta,\gamma)\in D} \tilde{Y}_t^{\tau}(\alpha,\beta,\gamma),$$ where $\tau\in\mathcal{R}_{t+}(\lambda)$,
    \begin{align}\label{FBSDE-Stopping-auxiliary}
\begin{aligned}
\tilde{Y}_{t}^\tau(\alpha,\beta,\gamma) &= {g({\tau}, {X}_{\tau})} + \int_{t}^{\tau} f^{\alpha,\beta,\gamma}(s, {X}_s, \tilde{Y}_s^\tau(\alpha,\beta,\gamma), \tilde{Z}_s^\tau(\alpha,\beta,\gamma), \tilde{C}_s^\tau(\alpha,\beta,\gamma))ds \\
&\qquad - \int_{t}^{\tau} \tilde{Z}_s^\tau(\alpha,\beta,\gamma) dW_s- \int_{t}^{\tau} \tilde{C}_s^\tau(\alpha,\beta,\gamma)\tilde{N}(ds),
\end{aligned}
\end{align}
and
{\small
\begin{align*}
D := \left\{\left. (\alpha,\beta,\gamma)\in L_\mathbb{G}^{2,p}(0, T;\mathbb{R}^{d+2})\right| (\alpha_t,\beta_t,\gamma_t)\in D(t,\omega)\,\text{compact in}\, \mathbb{R}\times\mathbb{R}^d\times[-\lambda,\infty)\right\}.
\end{align*}
}

Since the generator $f^{\alpha,\beta,\gamma}(t,X_t,y,z,c)$  of BSDE (\ref{FBSDE-Stopping-auxiliary}) is linear in $(y,z,c)$ with bounded coefficients, we apply \cref{{FBSDE-Penalization}} for the linear case and obtain
$$\hat{Y}_t^{\lambda}(\alpha,\beta,\gamma)=\sup_{\tau\in\mathcal{R}_{t+}(\lambda)}\tilde{Y}_{t}^\tau(\alpha,\beta,\gamma),$$ where
\begin{align}\label{FBSDE-Penalization-zc-apply}
\begin{aligned}
\hat{Y}_t^\lambda(\alpha,\beta,\gamma) &= g({ T}, {X}_T) + \int_t^{T}\Big[f^{\alpha,\beta,\gamma}(s, {X}_s, \hat{Y}_s^\lambda(\alpha,\beta,\gamma), \hat{Z}_s^\lambda(\alpha,\beta,\gamma), \hat{C}_s^\lambda(\alpha,\beta,\gamma))\\
&\qquad +\big(1+ \frac{\gamma_s}{\lambda}\big) \lambda \big(g(s,X_s) + {A}_s -\hat{Y}_{s}^\lambda(\alpha,\beta,\gamma) - \hat{C}_{s}^\lambda(\alpha,\beta,\gamma))^+\Big]  ds \\
&\qquad - \int_t^{ T} \hat{Z}_s^\lambda(\alpha,\beta,\gamma) dW_s - \int_t^{ T}  \hat{C}_s^\lambda(\alpha,\beta,\gamma)  \tilde{N}(ds),\quad t\in[0, T].
\end{aligned}
\end{align}

Observe that the generator of  \eqref{FBSDE-Penalization-zc-apply} satisfies
\cref{(H3)}. Indeed, for any $(t,y,z)\in [0,T]\times \mathbb{R}\times \mathbb{R}^d,\,c$ and $c^\prime\in \mathbb{R}$ with $c \neq c^\prime$, the generator meets the requirements of Assumption \ref{(H3)}:
\begin{align*}
\begin{aligned}
&\frac{\gamma_t c + (\lambda+\gamma_t) (g(t,X_t)+A_t-y-c)^+}{\lambda(c - c^\prime)}\\
&-\frac{\gamma_t c^\prime +(\lambda+\gamma_t) (g(t,X_t)+A_t-y-c^\prime)^+}{\lambda(c - c^\prime)} \geq -1,
\end{aligned}
\end{align*}
so the comparison theorem also holds for \eqref{FBSDE-Penalization-zc-apply}. In turn,
$$Y_t^{\lambda}=\sup_{(\alpha,\beta,\gamma)\in D}\hat{Y}_t^{\lambda}(\alpha,\beta,\gamma),$$
where
\begin{align*}
\begin{aligned}
{Y}_t^\lambda&= g({ T}, {X}_T) + \int_t^{T}\sup_{(\alpha,\beta,\gamma)\in D(s,\omega)}\Big[f^{\alpha,\beta,\gamma}(s, {X}_s, {Y}_s^\lambda, {Z}_s^\lambda, {C}_s^\lambda)\\
&\qquad +\big(1+ \frac{\gamma_s}{\lambda}\big) \lambda \big(g(s,X_s) + {A}_s -{Y}_{s}^\lambda - {C}_{s}^\lambda\big)^+\Big]  ds \\
&\qquad - \int_t^{ T} {Z}_s^\lambda dW_s - \int_t^{ T}  {C}_s^\lambda  \tilde{N}(ds),\quad t\in[0, T].
\end{aligned}
\end{align*}
For any $(s,y,z,c)\in[0,T]\times\mathbb{R}\times\mathbb{R}^d\times\mathbb{R}$, by the Fenchel-Moreau theorem,
\begin{align*}
&\sup_{(\alpha,\beta,\gamma)\in D(s,\omega)}\Big[f^{\alpha,\beta,\gamma}(s, {X}_s, y, z, c)+\big(1+ \frac{\gamma_s}{\lambda}\big) \lambda \big(g(s,X_s) + {A}_s -y - c\big)^+\Big]\\
=&\sup_{(\alpha,\beta,\gamma)\in D(s,\omega)}\Big[\alpha y+\beta z+\gamma (c+(g(s,X_s)+A_s-y-c)^+)-f^*(s,X_s,\alpha,\beta,\gamma)\Big]\\
&+\lambda (g(s,X_s)+A_s-y-c)^+\\
=&f\left(s,X_s,y,z,c+(g(s,X_s)+A_s-y-c)^+\right)+\lambda (g(s,X_s)+A_s-y-c)^+,
\end{align*}
which is the generator of PBSDE (\ref{FBSDE-Penalization}). Subsequently, due to the uniqueness of the solution to PBSDE (\ref{FBSDE-Penalization}), we conclude that $Y_t^{\lambda}=\hat{Y}_t^{\lambda}$. The desired result can then be inferred by observing that
\begin{align*}
\hat{Y}_t^{\lambda}&=\sup_{(\alpha,\beta,\gamma)\in D}\sup_{\tau\in\mathcal{R}_{t+}(\lambda)}\tilde{Y}_{t}^\tau(\alpha,\beta,\gamma)\\
&=\sup_{\tau\in\mathcal{R}_{t+}(\lambda)}\sup_{(\alpha,\beta,\gamma)\in D}\tilde{Y}_{t}^\tau(\alpha,\beta,\gamma)\\
&=\sup_{\tau\in\mathcal{R}_{t+}(\lambda)}\tilde{Y}_{t}^\tau
\end{align*}
\end{proof}

\section{Application to American option pricing in a nonlinear market with Poisson stopping constraints}
\label{sec:appli}

We conclude the paper by solving \cref{ep-2}. Recall that
\begin{align*}
f(t, y, z, c) = & -\underline{r}_t \left(y - \frac{z}{\sigma_t} - \frac{c}{\eta_t}\mathbb{I}_{\{\eta_t\neq 0\}}\right)^+
                  + \overline{r}_t \left(y - \frac{z}{\sigma_t} - \frac{c}{\eta_t}\mathbb{I}_{\{\eta_t\neq 0\}}\right)^- \\
                & - \underline{\mu}^1_t \frac{z^+}{\sigma_t}
                  + \overline{\mu}^1_t \frac{z^-}{\sigma_t}
                  - \underline{\mu}^2_t \left(\frac{c}{\eta_t}\right)^+\mathbb{I}_{\{\eta_t\neq 0\}}
                  + \overline{\mu}^2_t \left(\frac{c}{\eta_t}\right)^-\mathbb{I}_{\{\eta_t\neq 0\}},
\end{align*}
which is Lipchitz continuous and convex in $(y,z,c)$. Its convex dual has the expression
\begin{align*}
f^*(t,\alpha,\beta,\gamma)&=\sup_{(y,z,c)\in\mathbb{R}\times\mathbb{R}\times\mathbb{R}}\Big\{\alpha y+
    \beta z+\gamma c - f(t,y,z,c)\Big\}\\
    &=\delta_{\{-\overline{r}_t\leq \alpha\leq -\underline{r}_t,\ -\overline{\mu}^1_t\leq \alpha+\sigma_t\beta \leq -\underline{\mu}^1_t,\ -\overline{\mu}^2_t\leq \alpha+\eta_t\gamma\leq -\underline{\mu}^2_t\}}(\alpha,\beta,\gamma),
    \end{align*}
where $\delta_A(\cdot)$ is the characteristic function of set $A$, i.e. $\delta_A(x)=0$ for $x\in A$ and $\delta_A(x)=\infty$ for $x\notin A$.
Moreover, the effective domain of $f^*$ is
\[
D(t,\omega) = \left\{
(\alpha,\beta,\gamma) :
\begin{array}{l}
-\overline{r}_t \leq \alpha \leq -\underline{r}_t, \\
-\overline{\mu}^1_t \leq \alpha + \sigma_t \beta \leq -\underline{\mu}^1_t, \\
-\overline{\mu}^2_t \leq \alpha + \eta_t \gamma \leq -\underline{\mu}^2_t \ \text{for} \ \eta_t \neq 0
\end{array}
\right\}.
\]
To guarantee \cref{(H4)} holds, in particular, $\gamma\geq -\lambda$, we further assume that
\begin{equation}\label{parameter_condition}
\frac{\overline{\mu}_t^2 - \underline{r}_t}{\eta_t} \leq \lambda\ \text{for} \ \eta_t > 0; \quad \frac{\underline{\mu}_t^2 - \overline{r}_t}{\eta_t} \leq \lambda\ \text{for} \ \eta_t < 0;
\end{equation}
so the effective domain $D(t,\omega)$ is compact and satisfies \( D(t,\omega) \subset \mathbb{R} \times \mathbb{R} \times [-\lambda, \infty) \).
By applying \cref{result-zc-nonlinear}, the price of the American option in the nonlinear market with Poisson stopping constraints is given by $$\hat{Y}_t^{\lambda}=\sup_{\tau\in\mathcal{R}_{t+}(\lambda)}\tilde{Y}_t^{\tau},$$ where
{\small
\begin{align}\label{exampleBSDE}
\hat{Y}_t^\lambda = & \ g(S^1_T, S^2_T) + \int_t^{T} \Bigg\{
    -\underline{r}_s \left(\hat{Y}_s^{\lambda} - \frac{\hat{Z}_s^{\lambda}}{\sigma_s} - \frac{\hat{C}_s^{\lambda} + \left( g(S^1_s, S^2_s) + A_s - \hat{Y}_s^\lambda - \hat{C}_s^\lambda \right)^+}{\eta_s} \mathbb{I}_{\{\eta_s \neq 0\}} \right)^+\\
& \quad + \overline{r}_s \left(\hat{Y}_s^{\lambda} - \frac{\hat{Z}_s^{\lambda}}{\sigma_s} - \frac{\hat{C}_s^{\lambda} + \left( g(S^1_s, S^2_s) + A_s - \hat{Y}_s^\lambda - \hat{C}_s^\lambda \right)^+}{\eta_s} \mathbb{I}_{\{\eta_s \neq 0\}} \right)^- \notag \\
& \quad - \underline{\mu}_s^1 \frac{(\hat{Z}_s^{\lambda})^+}{\sigma_s} + \overline{\mu}_s^1 \frac{(\hat{Z}_s^{\lambda})^-}{\sigma_s} \notag \\
& \quad - \underline{\mu}_s^2 \left[\frac{\hat{C}_s^{\lambda} + \left( g(S^1_s, S^2_s) + A_s - \hat{Y}_s^\lambda - \hat{C}_s^\lambda \right)^+}{\eta_s}\right]^+ \mathbb{I}_{\{\eta_s \neq 0\}} \notag \\
& \quad + \overline{\mu}_s^2 \left[\frac{\hat{C}_s^{\lambda} + \left( g(S^1_s, S^2_s) + A_s - \hat{Y}_s^\lambda - \hat{C}_s^\lambda \right)^+}{\eta_s}\right]^- \mathbb{I}_{\{\eta_s \neq 0\}} \notag \\
& \quad + \lambda \left( g(S^1_s, S^2_s) + A_s - \hat{Y}_s^\lambda - \hat{C}_s^\lambda \right)^+
\Bigg\} ds \notag \\
& \quad - \int_t^{T} \hat{Z}_s^\lambda dW_s - \int_t^{T} \hat{C}_s^\lambda \tilde{N}(ds), \quad t \in [0, T],\notag
\end{align}
}
with the optimal stopping time given by
\begin{align*}
\hat{\tau}^*_{t} := \inf \Big\{T_n >t \,\Big|\, \hat{Y}^\lambda_{T_n} \leq {g(S^1_{T_n}, S^2_{T_n})}\Big\} \wedge T.
\end{align*}

Next, we examine a specific example of an American put option with staircase-style strike prices and bid-ask interest rates. Consider the payoff function in the form of \( g(S^1, S^2) = (S^2 - S^1)^+ \), where
\[
dS_t^1 = S_t^1 [ (\underline{\mu}^1 \mathbb{I}_{\pi_t^1 \geq 0} - \overline{\mu}^1 \mathbb{I}_{\pi_t^2 < 0}) dt + \sigma dW_t], \quad S_0^1 = s,
\]
for \(\overline{\mu}^1 \geq \underline{\mu}^1\) and \(\sigma > 0\) are constants. Moreover,
\[
dS_t^2 = S_{t-}^2 [\lambda \eta_t dt + \eta_t \tilde{N}(dt)], \quad S_0^2 = 1,
\]
for \(\eta_t = \eta \mathbb{I}_{\{t \leq T_k\}}\) for some integer \( k \geq 1 \), where \(\eta > -1\) is a non-zero constant.
Under the above assumptions,
\[
S_t^2 = \prod_{0<s\leq t}(1+\eta_s\Delta N_s) = \sum_{i=1}^{k}(1+\eta)^{i-1}\mathbb{I}_{t}^i + (1+\eta)^k \mathbb{I}_{\{t \geq T_k\}},
\]
and the payoff can be decomposed into a series of put payoffs with staircase-style strike prices
\[
g(S_t^1, S_t^2) = \sum_{i=1}^k \left[(1+\eta)^{i-1} - S_t^1 \right]^+ \mathbb{I}_{t}^i + \left[(1+\eta)^k - S_t^1 \right]^+ \mathbb{I}_{\{t \geq T_k\}}.
\]


Both the bid and ask interest rates on risk-free bonds are constant, with $\overline{r}_t = \overline{r}$ and $\underline{r}_t = \underline{r}$, where $\overline{r}$ and $\underline{r}$ are two constant values such that $\overline{r} \geq \underline{r}$. For the conditions in (\ref{parameter_condition}) to hold, we further assume that
\begin{equation}\label{parameter_condition_2}
\underline{r} \geq 0 \quad \text{for} \ \eta > 0; \quad \overline{r} \leq 0 \quad \text{for} \ \eta < 0.
\end{equation}

Applying the Jacod-Pham decomposition, we obtain
\begin{align}\label{Example-decom}
\hat{Y}_t^{\lambda} = \sum_{i=0}^{k-1}\hat{Y}_t^{i,\lambda}(T_{(i)})\mathbb{I}_{t}^{i+1} + \hat{Y}_t^{k,\lambda}(T_{(k)})\mathbb{I}_{\{t \geq T_k\}},
\end{align}
where $\hat{Y}^{i,\lambda}(T_{(i)}) = \hat{Y}^{i,\lambda}(\theta_{(i)})|_{\theta_i=T_i}$, for $0\leq i\leq k$, solve the Brownian motion driven BSDEs:
{\small
\begin{align*}
\hat{Y}_t^{k,\lambda}(\theta_{(k)}) = & \ [(1+\eta)^k - S_T^1]^+ + \int_t^{T} \Bigg\{
    -\underline{r}\left(\hat{Y}_s^{\lambda}(\theta_{(k)}) - \frac{\hat{Z}_s^{\lambda}(\theta_{(k)})}{\sigma}\right)^+ \\
& \quad + \overline{r}\left(\hat{Y}_s^{\lambda}(\theta_{(k)}) - \frac{\hat{Z}_s^{\lambda}(\theta_{(k)})}{\sigma}\right)^- \notag \\
& \quad - \underline{\mu}^1 \frac{(\hat{Z}_s^{\lambda}(\theta_{(k)}))^+}{\sigma} + \overline{\mu}^1 \frac{(\hat{Z}_s^{\lambda}(\theta_{(k)}))^-}{\sigma} \notag \\
& \quad + \lambda \left([(1+\eta)^k - S_s^1]^+ - \hat{Y}_s^{k,\lambda}(\theta_{(k)})\right)^+
\Bigg\} ds \notag \\
& \quad - \int_t^{T} \hat{Z}_s^\lambda(\theta_{(k)}) dW_s, \quad t \in [\theta_{k} \wedge T, T], \notag
\end{align*}
}
and for $0 \leq i \leq k-1$,
{\small
\begin{align*}
\hat{Y}_t^{i,\lambda}(\theta_{(i)}) = & \ [(1+\eta)^i - S_T^1]^+ + \int_t^{T} \Bigg\{
    -\underline{r}\left(\hat{Y}_s^{\lambda}(\theta_{(i)}) - \frac{\hat{Z}_s^{\lambda}(\theta_{(i)})}{\sigma} \right. \\
& \left. - \frac{Y_s^{i+1,\lambda}(\theta_{(i+1)}) - Y_s^{i,\lambda}(\theta_{(i)}) + \{[(1+\eta)^{i+1} - S^1_{s}]^+ - Y_s^{i+1,\lambda}(\theta_{(i+1)})\}^+}{\eta}\right)^+ \\
& \quad + \overline{r}\left(\hat{Y}_s^{\lambda}(\theta_{(i)}) - \frac{\hat{Z}_s^{\lambda}(\theta_{(i)})}{\sigma} \right. \\
& \left. - \frac{Y_s^{i+1,\lambda}(\theta_{(i+1)}) - Y_s^{i,\lambda}(\theta_{(i)}) + \{[(1+\eta)^{i+1} - S^1_{s}]^+ - Y_s^{i+1,\lambda}(\theta_{(i+1)})\}^+}{\eta}\right)^- \notag \\
& \quad - \underline{\mu}^1 \frac{(\hat{Z}_s^{\lambda}(\theta_{(i)}))^+}{\sigma} + \overline{\mu}^1 \frac{(\hat{Z}_s^{\lambda}(\theta_{(i)}))^-}{\sigma} \Bigg\} ds \\
& \quad - \int_t^{T} \hat{Z}_s^\lambda(\theta_{(i)}) dW_s, \quad t \in [\theta_{i} \wedge T, T]. \notag
\end{align*}
}

The above sequence of BSDEs are Markovian due to the constant parameter assumptions. Hence, by the Feynman-Kac formula:
\begin{align*}
\hat{Y}_t^{i,\lambda}( \theta_{(i)}) = \hat{V}^{i,\lambda}(t,S^1_t),\quad \hat{Z}_t^{i,\lambda}( \theta_{(i)}) = \sigma S_t^1 \partial_s\hat{V}^{i,\lambda}(t,S^1_t),\quad 0\leq i\leq k,
\end{align*}
where
\begin{equation}
\left\{
\begin{aligned}
\partial_t \hat{V}^{k,\lambda}(t,s) +& \frac{1}{2}\sigma^2 s^2 \partial_{ss}^2 \hat{V}^{k,\lambda}(t,s) - \underline{r}\left(\hat{V}^{k,\lambda}(t,s) - s \partial_s \hat{V}^{k,\lambda}(t,s)\right)^{+} \\
& + \overline{r}\left(\hat{V}^{k,\lambda}(t,s) - s \partial_s \hat{V}^{k,\lambda}(t,s)\right)^{-} \\
& + \lambda\left([(1+\eta)^k - s]^+ - \hat{V}^{k,\lambda}(t,s)\right)^{+} = 0, \\
\hat{V}_T^{k,\lambda}(T,s) = & [(1+\eta)^k - s]^+,
\end{aligned}
\right.
\label{eq:V_k_lambda}
\end{equation}
and for \(0 \leq i \leq k-1\),
\begin{equation}
\left\{
\begin{aligned}
\partial_t \hat{V}^{i,\lambda}(t,s) +& \frac{1}{2}\sigma^2 s^2 \partial_{ss}^2 \hat{V}^{i,\lambda}(t,s) \\
& - \underline{r} \left( \hat{V}^{i,\lambda}(t,s) - s \partial_s \hat{V}^{i,\lambda}(t,s) \right. \\
& \left. - \frac{\hat{V}^{i+1,\lambda}(t,s) - \hat{V}^{i,\lambda}(t,s) + \{[(1+\eta)^{i+1} - s]^+ - \hat{V}^{i+1,\lambda}(t,s)\}^+}{\eta} \right)^{+} \\
& + \overline{r} \left( \hat{V}^{i,\lambda}(t,s) - s \partial_s \hat{V}^{i,\lambda}(t,s) \right. \\
& \left. - \frac{\hat{V}^{i+1,\lambda}(t,s) - \hat{V}^{i,\lambda}(t,s) + \{[(1+\eta)^{i+1} - s]^+ - \hat{V}^{i+1,\lambda}(t,s)\}^+}{\eta} \right)^{-} \\
& = 0, \\
\hat{V}_T^{i,\lambda}(T,s) = & [(1+\eta)^i - s]^+.
\end{aligned}
\right.
\label{eq:V_i_lambda}
\end{equation}

For the recursive sequence of PDEs (\ref{eq:V_k_lambda})-(\ref{eq:V_i_lambda}), the final terms in these equations have natural interpretations, reflecting the decision-making process involving stopping and continuing opportunities. Indeed, in PDE (\ref{eq:V_k_lambda}) for \(t \geq T_k\), the term
\begin{align*}
&\lambda \max\left([(1+\eta)^k - s]^+ - \hat{V}^{k,\lambda}(t,s), 0\right)\\
=&\lambda \max\left([(1+\eta)^k - s]^+ - \hat{V}^{k,\lambda}(t,s), \hat{V}^{k,\lambda}(t,s)-\hat{V}^{k,\lambda}(t,s)\right)
\end{align*}
captures the player's choice between two actions. The player can either stop to secure an immediate payoff of \([(1+\eta)^k - s]^+\), or continue, in which case its own value function \(\hat{V}^{k,\lambda}(t,s)\) is used to determine the outcome.The transition rate
$\lambda$ reflects the cost of postponing a decision, with higher values indicating more frequent transitions between states due to higher postponement costs. This term is consistent with the penalty term in the penalized BSDE (\ref{FBSDE-Penalization-2}).

For PDE (\ref{eq:V_i_lambda}), valid for \(t \in [T_{i-1}, T_i)\), when \(\underline{r} = \overline{r}\), the fraction term simplifies to
\begin{align*}
&\overline{r}\frac{\hat{V}^{i+1,\lambda}(t,s) - \hat{V}^{i,\lambda}(t,s) + \{[(1+\eta)^{i+1} - s]^+ - \hat{V}^{i+1,\lambda}(t,s)\}^+}{\eta}\\
=&\frac{\overline{r}}{\eta} \max\left\{[(1+\eta)^{i+1} - s]^+ - \hat{V}^{i,\lambda}(t,s), \hat{V}^{i+1,\lambda}(t,s) - \hat{V}^{i,\lambda}(t,s)\right\}.
\end{align*}
This term represents the player's choice between two competing outcomes: either stopping to receive the payoff \([(1+\eta)^{i+1} - s]^+\) or continuing to the next interval where the value function \(\hat{V}^{i+1,\lambda}(t,s)\) is evaluated. The factor \(\frac{\overline{r}}{\eta}\) scales the impact of the decision to stop and continue, reflecting the transition rate at the next decision point.

The optimal stopping regions will depend on Poisson stopping intervals with \( k \) different cases. For \( t \geq T_{k-1} \), the optimal stopping time is defined as
\begin{align*}
\hat{\tau}^{*}_{t} :=& \inf \left\{ T_n > t \,\Big|\, \hat{V}^{k,\lambda}(T_n, S_{T_n}) \leq [(1+\eta)^k - S_{T_n}]^+ \right\} \wedge T,
\end{align*}
and the stopping region is defined as
\[
\{ (t, s) \in [T_{k-1} \wedge T, T] \times \mathbb{R}_+ \mid \hat{V}^{k,\lambda}(t, s) = [(1+\eta)^k - s]^+ \}.
\]

In general, for \( t \in [T_{i-1}, T_i) \), \( 1\leq i \leq k-1 \), the optimal stopping time is defined as
\begin{align*}
\hat{\tau}^{*}_{t} :=& \inf \left\{ T_n > t \,\Big|\, \hat{V}^{i,\lambda}(T_n, S_{T_n}) \leq [(1+\eta)^i - S_{T_n}]^+ \right\} \wedge T,
\end{align*}
and the stopping region is defined as

\[
\{ (t, s) \in [T_{i-1} \wedge T, T] \times \mathbb{R}_+ \mid \hat{V}^{i,\lambda}(t, s) = [(1+\eta)^i - s]^+ \}.
\]

In the following, using the fully implicit finite difference method, we provide numerical analysis for the above nonlinear American option pricing model. Unless otherwise stated, the following parameters are used: $\sigma=0.2, \eta=0.1, T = 5, \underline{r}=0.02, \overline{r}=0.05, k=4$.

[Insert Figure \ref{fig:values} here]

Figure \ref{fig:values} demonstrates the asymptotic behaviours of value functions $\hat{V}^{k,\lambda} (k=4)$ and the corresponding free boundaries with respect to $\lambda$. Intuitively, the value functions increases as $\lambda$ rises, as a larger $\lambda$ means more exercise opportunities for the option holder. The left panel of Figure \ref{fig:values} confirms this point. Notably, the left panel indicates the smooth pasting effect when $\lambda$ is very large. This observation implies that the $\hat{V}^{k,\lambda}$ converges to a value function of an optimal stopping model as $\lambda$ goes to $\infty$, which is consistent with conventional American option pricing.
The right panel of Fig \ref{fig:values} illustrates the impact of $\lambda$ on the exercise boundary (free boundary) associated with $\hat{V}^{k,\lambda}$. The exercise boundary divides the whole region into two parts: the upper region is the holding region, and the lower region is the exercise region. In addition to the convergence of the exercise boundary with respect to $\lambda$, this panel shows that the exercise region shrinks as $\lambda$ increases. This phenomenon suggests that more exercise opportunities actually diminish the likelihood of exercising for option holders.

[Insert Figure \ref{fig:simulations} here]

Figure \ref{fig:simulations} compares the value functions and exercise boundaries for different \(i\) values, ranging from \(0\) to \(k\) (\(k=4\)). The upper panel displays the value functions \(\hat{V}^{i,\lambda}\) for \(i = 0\) to \(4\). Notably, the larger the \(i\), the higher the value function, with \(\hat{V}^{0,\lambda}\) being the smallest.
The bottom panel shows a simulation of the optimal stopping time. It presents four exercise boundaries: the lowest corresponds to \(T_1\), the second lowest to \(T_2\), the third lowest to \(T_3\), and the highest to \(T_4\) and beyond. In the simulated sample path, the optimal stopping time is at \(T_4\).

\section{Conclusion and extensions}
\label{sec:concl}
This paper advances the understanding of the recursive optimal stopping problem with Poisson stopping constraints, especially within the context of American option pricing in a nonlinear market. The key methodology employed is the Jacod-Pham decomposition, which enables the derivation of a novel PBSDE representation for the value function and its corresponding optimal stopping time.

{\gl Several extensions of our results are possible. First, extending from the finite horizon setting to the {infinite horizon case} is straightforward by letting \( T = \infty \). However, a {monotonicity condition} on the generator is required to ensure well-posedness of BSDEs, namely,
\begin{equation*}
(f(t,x,y,z,c) - f(t,x,y',z,c)) \cdot (y - y') \leq \lambda (y - y')^2, \quad \lambda < 0,
\end{equation*}
as considered in \cite{Hu2020}. Under this condition, the main representation result in Theorem \ref{result-zc-nonlinear} remains valid.

Another natural extension is from the {recursive optimal stopping} framework in this paper to {recursive optimal switching}, which was first introduced in \cite{liang2015stochastic} for the linear case. Consider the recursive objective functional:
\begin{align*}
{Y}_t^u &= g^{u_T}(T, X_T) + \int_t^T f^{u_s}\left(s, X_s, {Y}_s^u, {Z}_s^u, C_s^u\right) ds
- \int_t^T Z_s^u d W_s - \int_t^T C_s^u \tilde{N}(d s) \\
&\quad - \sum_{t < T_m \leq T} \Delta_{T_m}^{u_{T_m-}, u_{T_m}},
\quad t \in [0, T], \quad u \in \mathcal{K}_0^i(\lambda),
\end{align*}
where \( \Delta^{i,j} \), \( 1 \leq i, j \leq k \), represents the {switching cost} from regime \( i \) to regime \( j \), and the {admissible impulse control set} \( \mathcal{K}_0^i(\lambda) \) is given by
\begin{equation*}
\mathcal{K}_0^i(\lambda) := \left\{ u \in L_{\mathbb{G}}^2(0, T ;\{1,2, \dots, k\}) \mid u_s = i \mathbb{I}_{\{s < T_1\}} + \sum_{m \geq 1} a_m \mathbb{I}_s^{m+1}, a_m \in \mathcal{G}_{T_m} \right\}.
\end{equation*}
Then, similar to Theorem \ref{result-zc-nonlinear}, we obtain the following representation result:
\begin{equation*}
\hat{Y}_0^i = \sup_{u \in \mathcal{K}_0^i} Y_0^u,
\end{equation*}
where \( \hat{Y}^i \) satisfies the {multidimensional PBSDE}:
\begin{align*}
\hat{Y}_t^i &= g^i(T,X_T)
+ \int_t^T \Bigg[ f^i \Big(s, X_s, \hat{Y}_s^i, \hat{Z}_s^i, \hat{C}_s^i
+ \big(\mathcal{M} \hat{Y}_s^i + \mathcal{A} \hat{Y}_s^i - \hat{Y}_s^i - \hat{C}_s^i\big)^{+} \Big) \\
&\quad + \lambda \big(\mathcal{M} \hat{Y}_s^i + \mathcal{A} \hat{Y}_s^i
- \hat{Y}_s^i - \hat{C}_s^i\big)^{+} \Bigg] ds \\
&\quad - \int_t^T \hat{Z}_s^i dW_s - \int_t^T \hat{C}_s^i \tilde{N}(ds),
\quad t \in [0, T], \quad i \in \{1,2, \dots, k\}.
\end{align*}
with the {impulse control operator}:
\begin{equation}
\mathcal{M} \hat{Y}_t^i := \max _{j \neq i} \left\{\hat{Y}_t^j - \Delta_t^{i, j} \right\}.
\end{equation}
Using the {Jacod-Pham decomposition}, this can be rewritten as
\begin{equation}
\mathcal{M} \hat{Y}_t^i = \sum_{j \geq 0} \mathcal{M} \hat{Y}_t^{j, i}(T_{(j)}) \mathbb{I}_t^{j+1},
\end{equation}
where \( \mathcal{M} \hat{Y}_t^{j, i} (T_{(j)}) \in L_{\mathbb{F}}^2(0, T ; \mathbb{R}) \) for \( \theta_{(j)} \in \Delta_j \).  The auxiliary process \( \mathcal{A} \hat{Y}^i \) is then defined as
\begin{equation}
\mathcal{A} \hat{Y}_t^i := \mathcal{A} \hat{Y}_t^{0, i} \mathbb{I}_{\{t < T_1\}} + \sum_{j \geq 1} \mathcal{A} \hat{Y}_t^{j, i}(T_{(j)}) \mathbb{I}_t^{j+1},
\end{equation}
where
\begin{equation}
\mathcal{A} \hat{Y}_t^{j, i}(T_{(j)}) := \left. \mathcal{A} \hat{Y}_t^{j, i}(\theta_{(j)}) \right|_{\theta_{(j)}=T_{(j)}},
\end{equation}
with
\begin{equation}
\mathcal{A} \hat{Y}_t^{j, i}(\theta_{(j)}) := \mathcal{M} \hat{Y}_t^{j+1, i}(\theta_{(j)}, t) - \mathcal{M} \hat{Y}_t^{j, i}(\theta_{(j)}), \quad j \geq 0.
\end{equation}

We leave the detailed proofs of these two extensions to the interested reader.

Finally, an intriguing yet more {challenging} extension would be to consider the {risk-sensitive case}, where the PBSDE transforms into a {quadratic BSDE with jumps}. Such a formulation arises naturally in {stochastic control under risk-sensitive criteria}, and its study is significantly more involved. For related work on {risk-sensitive Poisson Dynkin games}, see \cite{Du2025}.  }


\section*{Acknowledgment}

We are grateful to the editor and the referee whose
constructive comments led to a substantial
enhancement of the paper.

\bibliographystyle{plain}


\begin{figure}[H]
    \centering
    \begin{subfigure}[t]{0.9\linewidth}
        \centering
        \includegraphics[width=1\textwidth]{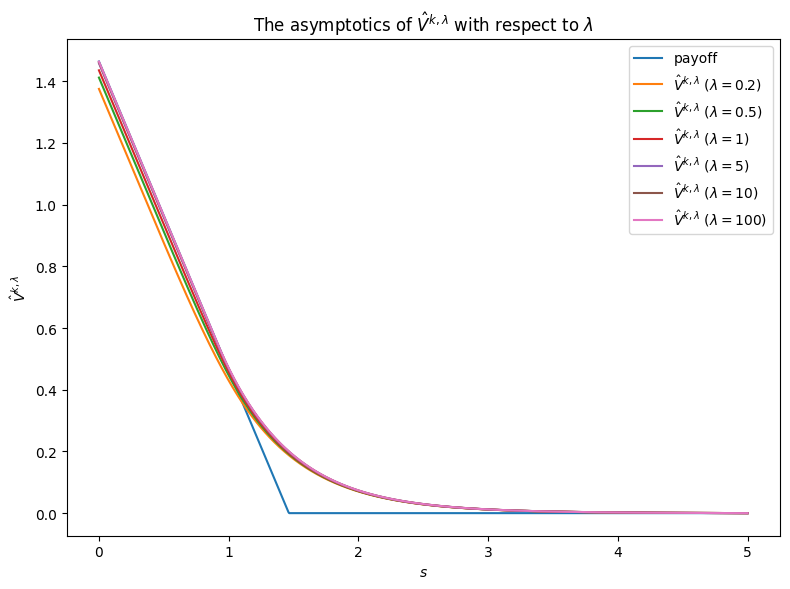}
    \end{subfigure}
    \vspace{0.01\textwidth}
    \begin{subfigure}[t]{0.9\linewidth}
        \centering
        \includegraphics[width=1\textwidth]{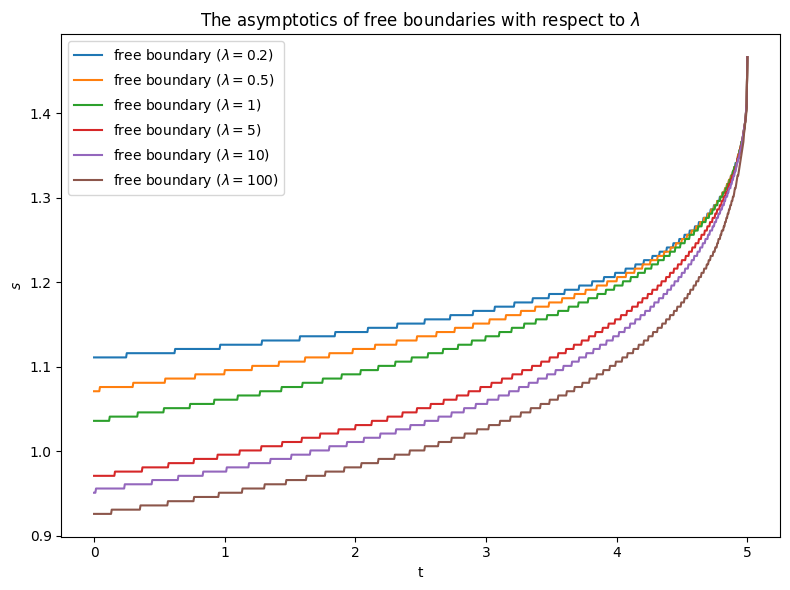}
    \end{subfigure}
    \caption{The asymptotic behaviours of the value functions and the free boundaries with respect to $\lambda$. Parameters: $\sigma = 0.2, \underline{r} = 0.02, \overline{r} = 0.05, \eta= 0.1, T = 5, k=4.$}
    \label{fig:values}
\end{figure}

\begin{figure}[H]
\centering
\begin{subfigure}[t]{0.9\linewidth}
\centering
\includegraphics[width=1\textwidth]{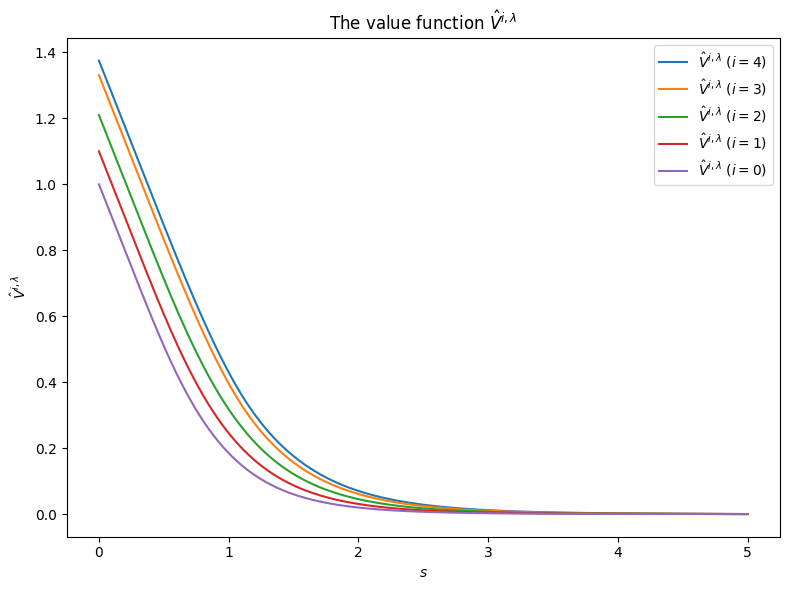}
\end{subfigure}
\vspace{0.02\textwidth}
\begin{subfigure}[t]{0.9\linewidth}
\centering
\includegraphics[width=1\textwidth]{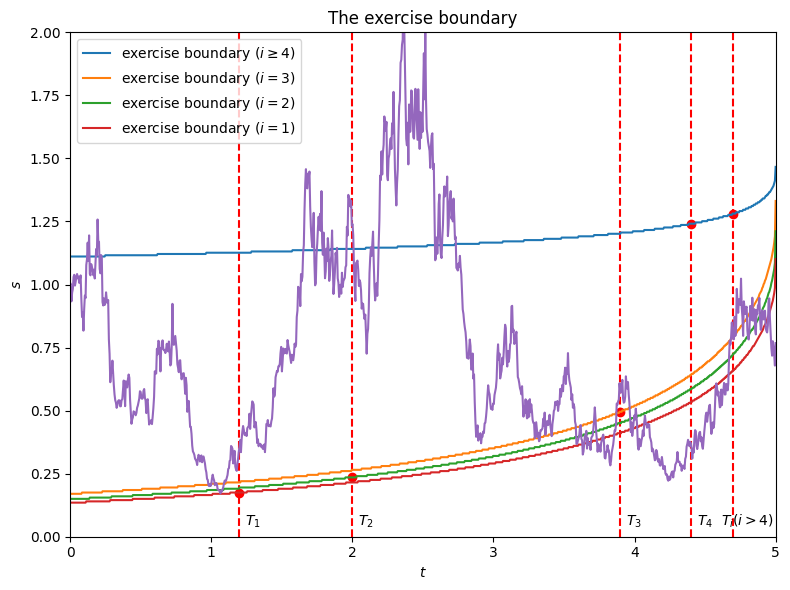}
\end{subfigure}
\caption{The value functions and the free boundaries at Poisson random times. Parameters: $\sigma = 0.2, \underline{r} = 0.02, \overline{r} = 0.05, \eta= 0.1, T = 5, k=4, \lambda=0.2.$}
\label{fig:simulations}
\end{figure}

\end{document}